\documentclass{amsart}
\usepackage{amsmath}
\usepackage{amssymb}
\usepackage{mathrsfs}
\usepackage{stmaryrd}
\usepackage{mathtools}
\usepackage{nicematrix}
\usepackage{tikz}
\usepackage[all]{xy}
\usepackage{colortbl}
\usepackage[utf8]{inputenc}
\usepackage{enumerate}
\usepackage{a4wide}
\usepackage{tikz-cd}

\usepackage[colorlinks=true]{hyperref}
\newtheorem{thmvoid}{}[section]

\newtheorem{theorem}[thmvoid]{Theorem}
\newtheorem{corollary}[thmvoid]{Corollary}
\newtheorem{lemma}[thmvoid]{Lemma}
\newtheorem{proposition}[thmvoid]{Proposition}
\newtheorem{example}[thmvoid]{Example}
\newtheorem{question}[thmvoid]{Question}

\theoremstyle{definition}

\newtheorem{definition}[thmvoid]{Definition}
\newtheorem{remark}[thmvoid]{Remark}

\numberwithin{equation}{section}

\title[Ramified covers of abelian varieties over torsion fields]{Ramified covers of abelian varieties over torsion fields}

\author{Lior Bary-Soroker}
\address{School of Mathematical Sciences, Tel Aviv University,
Tel Aviv 69978, Israel}
\email{barylior@tauex.tau.ac.il}

\author{Arno Fehm}
\address{Institut f\"{u}r Algebra, Technische Universit\"{a}t Dresden, 01062 Dresden, Germany}
\email{arno.fehm@tu-dresden.de}

\author{Sebastian Petersen}
\address{Institut für Mathematik, Universit\"at Kassel, 34121 Kassel, Germany}
\email{petersen@mathematik.uni-kassel.de}


\begin{document}

\newcommand{\Spec}{\mathrm{Spec}}
\newcommand{\Mor}{\mathrm{Mor}}
\newcommand{\Aut}{\mathrm{Aut}}
\newcommand{\Gal}{\mathrm{Gal}}
\newcommand{\Ram}{\mathrm{Ram}}
\newcommand{\codim}{\mathrm{codim}}
\newcommand{\mfm}{\mathfrak{m}}
\newcommand{\OOO}{\mathscr{O}}

\maketitle

\begin{center}
{\em Dedicated to Moshe Jarden on the occasion of his 80th birthday}
\end{center}

\begin{abstract}
We study rational points on ramified covers of abelian varieties over certain infinite Galois extensions of $\mathbb{Q}$.
In particular, we prove that 
every elliptic curve $E$ over $\mathbb{Q}$
has the weak Hilbert property
of Corvaja--Zannier
both over
the maximal abelian extension $\mathbb{Q}^{\rm ab}$ of $\mathbb{Q}$,
and over the field $\mathbb{Q}(A_{\rm tor})$
obtained by adjoining to $\mathbb{Q}$
all torsion points of some abelian variety $A$ over $\mathbb{Q}$.
\end{abstract}

\section{Introduction}

\noindent
Hilbert's irreducibility theorem, which is one of the fundamental results about the arithmetic of number fields, can be viewed as a statement about rational points on finite covers of the projective line $\mathbb{P}^1$,
which led to the definition of the Hilbert property of varieties:

\begin{definition}[\cite{Serre,CTS}]
A variety\footnote{all varieties are assumed integral, cf.~Definition \ref{def:var}} $X$ over a field $K$ of characteristic zero
has the \emph{Hilbert property} HP if for every finite collection of finite %separable
surjective morphisms $(\pi_i \colon Y_i \to X)_{i=1}^n$ with each $Y_i$ a normal variety over $K$ and ${\rm deg}(\pi_i)\geq 2$, the set $X(K) \smallsetminus \bigcup_{i=1}^n \pi_i(Y_i(K))$ is Zariski-dense in $X$.
\end{definition}

Since Hilbert's irreducibility theorem is of great importance in Galois theory and diophantine geometry, 
the Hilbert property has been investigated
intensively in recent years, see \cite{BFP,Borovoi,CZ,Coccia,Demeio,Demeio2,NakaharaStreeter,Streeter,Javanpeykar,BG,Javanpeykar22}.

A nonzero abelian variety $A$ over a number field $K$ 
does not have HP,
since $2A(K)$ has finite index in $A(K)$ by the weak Mordell--Weil theorem
(and so $A(K)=\bigcup_{i=1}^m\pi_i(A(K))$ where
$\pi_i:A\rightarrow A$, $x\mapsto 2x+P_i$ with
$P_1,\dots,P_m$ a set of representatives of $A(K)/2A(K)$).
After foundational work in \cite{DZ,Zannier} on algebraic groups,
Corvaja and Zannier in \cite{CZ}
explained this failure of HP for abelian varieties more generally by the existence of unramified covers.
They then asked whether abelian varieties satisfy the following weaker property that suffices for many applications:

\begin{definition}[\cite{CZ}]\label{def:WHP}
A smooth proper variety $X$ over a field $K$ of characteristic zero
has the \emph{weak Hilbert property} WHP if for every finite collection of finite  surjective \underline{ramified} morphisms $(\pi_i \colon Y_i \to X)_{i=1}^n$ with each $Y_i$ a  normal variety over $K$, %and ${\rm deg}(\pi_i)\geq 2$, 
the set $X(K) \smallsetminus \bigcup_{i=1}^n \pi_i(Y_i(K))$ is Zariski-dense in $X$.
\end{definition}

This was then answered positively in \cite{CDJLZ}:

\begin{theorem}[Corvaja--Demeio--Javanpeykar--Lombardo--Zannier 2020]\label{thm:CZ}
Every  abelian variety $A$ over a number field $K$
with $A(K)$ Zariski-dense 
has WHP.
\end{theorem}

Let $K$ be a field of characteristic zero.
It is known that if $X$ has HP or WHP over $K$
and $L/K$ is finite, then the base change $X_L$ has HP resp.~WHP
(see~\cite[Prop.~12.3.3]{FJ} resp.~\cite[Prop.~3.15]{CDJLZ}),
and one can see that at least in the case of HP this holds also for $L/K$ Galois with finitely generated, or more generally small, Galois group (cf.~\cite[Prop.~16.11.1]{FJ}).
There are many results in the literature
regarding $\mathbb{P}^1$ having HP (or equivalently WHP, as $\mathbb{P}^1$ is simply connected)
over other infinite algebraic extensions $L$ of $\mathbb{Q}$ --
fields $L$ for which $\mathbb{P}^1_L$ has HP 
are called {\em Hilbertian} 
following \cite{Lang_diophantine_geometry}, see for example \cite{Kuyk}
for abelian extensions and
\cite{Haran,BFW} for two different
generalizations of this,
and \cite[Chapters 12-13]{FJ} for an extensive discussion of Hilbertian fields.
Even more precise results are known
for example for linear algebraic groups like $\mathbb{G}_m$ (see e.g.~\cite[Theorem 2.1]{Zannier}).
However,
as remarked in \cite{Zannier},
no such results are available for abelian varieties.
Starting from Theorem~\ref{thm:CZ},
and developing further the methods of \cite{Haran},
we prove in Section \ref{sec:abelian} the first such result over infinite (non-small) Galois extensions of $\mathbb{Q}$ of number theoretic interest:

\begin{theorem}\label{thm:intro_abelian}
Let $K$ be a finitely generated field 
of characteristic zero and $A$ an abelian variety over $K$.
Let $L/K$ be an abelian extension
such that 
${\rm dim}_\mathbb{Q}(A_0(L)\otimes_\mathbb{Z}\mathbb{Q})=\infty$
for every
nonzero homomorphic image $A_0$ of $A_L$.
Then $A_L$ has WHP.
\end{theorem}

We remark 
that the assumption 
${\rm dim}_\mathbb{Q}(A_0(L)\otimes_\mathbb{Z}\mathbb{Q})=\infty$
seems natural in that
if $L=K^{\rm ab}$ is 
the maximal abelian extension of $K$
and $A(L)=A(K)$,
then $A_L$ does {\em not} have WHP (see Remark~\ref{rem:counterexample}).
Moreover,
the so-called Frey--Jarden conjecture,
named after a problem posed in  \cite{FreyJarden}, 
predicts that
in the case $K=\mathbb{Q}$
and $L=\mathbb{Q}^{\rm ab}$, 
every abelian variety
$A$ over $\mathbb{Q}$
(and then in fact also every abelian variety
$A$ over $K^{\rm ab}$ for a number field $K$,~cf.~Lemma~\ref{lem:rank_res})
satisfies this condition.

\begin{corollary}\label{cor:intro_Frey_Jarden}
Let $K$ be a number field and
assume that the Frey--Jarden conjecture holds
(i.e.~Question \ref{q:FreyJarden} has a positive answer).
Then every abelian variety $A$ over $K^{\rm ab}$ 
has WHP.
\end{corollary}

By \cite[Theorem 1.2]{FehmPetersen}, the conjecture
would follow for example if $\mathbb{Q}^{\rm ab}$ is {\em ample} (or {\em large}), cf.~\cite[\S3.3]{BFsurvey}.
See  \cite{Petersen,SY,ImLarsen} and the references therein for progress on the Frey--Jarden conjecture for specific abelian varieties. 
Since the conjecture holds 
for elliptic curves over $\mathbb{Q}$  (as proven already in \cite{FreyJarden}, cf.~Proposition \ref{prop:FreyJarden}),
as a special case of Theorem \ref{thm:intro_abelian}
we obtain:

\begin{corollary}\label{cor:intro_E}
Let $E$ be an elliptic curve over $\mathbb{Q}$.
Then $E_{\mathbb{Q}^{\rm ab}}$ has WHP.
\end{corollary}

While by the Kronecker--Weber theorem, $\mathbb{Q}^{\rm ab}$
can be seen as the field obtained by adjoining all torsion points
of the linear algebraic group $\mathbb{G}_m$ to $\mathbb{Q}$,
our second result
concerns the extension $\mathbb{Q}(A_{\rm tor})$
obtained from $\mathbb{Q}$
by adjoining all torsion points of an abelian variety $A$ over $\mathbb{Q}$.
The fact that $\mathbb{P}_{\mathbb{Q}(A_{\rm tor})}^1$ has HP
was first proven by Jarden \cite{Jarden}
(which inspired for example the above mentioned \cite{BFW}).
In Section \ref{sec:tor} we combine
Theorem \ref{thm:CZ} with the method of \cite{Jarden} to obtain:

\begin{theorem}\label{thm:intro_E_tor}
Let $A$ be an abelian variety over $\mathbb{Q}$ and $E$
an elliptic curve over $\mathbb{Q}$.
Then $E_{\mathbb{Q}(A_{\rm tor})}$
has WHP.
\end{theorem}

This can be seen as a partial result
towards a conjecture by Zannier
\cite[\S2]{Zannier}
regarding 
{\em torsion} points of $A$
coming from
$\mathbb{Q}(A_{\rm tor})$-rational points on ramified covers of $A$.
In a different direction
we remark that by a result of Larsen \cite{Larsen}
(see also \cite[Corollary 1.2]{Habegger}, \cite[Lemma A.7]{BHP}),
$E(\mathbb{Q}(E_{\rm tor}))\cong E_{\rm tor}\oplus\mathbb{Z}^\omega$,
and so $2E(\mathbb{Q}(E_{\rm tor}))$
has {\em infinite} index in $E(\mathbb{Q}(E_{\rm tor}))$,
which eliminates the above obstruction to HP for $E$ over a number field.
This suggests the following question:

\begin{question}
Let $E$ be an elliptic curve over $\mathbb{Q}$.
Does $E_{{\mathbb{Q}}(E_{\rm tor})}$ have HP?
\end{question}

We conclude this introduction 
by saying that in fact \cite{CDJLZ} proves a more precise statement than Theorem \ref{thm:CZ},
and it is this stronger statement that we use in our proofs.
This stronger statement however does not carry over to abelian varieties over $\mathbb{Q}^{\rm ab}$ or $\mathbb{Q}(A_{\rm tor})$, 
by the above mentioned Remark \ref{rem:counterexample}.
We point out that our results are specific to
abelian varieties and the extensions
$\mathbb{Q}^{\rm ab}$
and $\mathbb{Q}(A_{\rm tor})$,
and it is not immediately clear how to extend them to other varieties or other extensions.
For example we even do not know whether if $X$ is a $K$-variety with WHP
and $L/K$ is a Galois extension with finitely generated Galois group,
then $X_L$ has WHP.

We begin with some preliminaries in Section \ref{sec:prelim}.
%,
%before we then i
In Section \ref{sec:sketch}
we
describe the strategy of our proofs of
Theorems~\ref{thm:intro_abelian} and \ref{thm:intro_E_tor}
and provide a few lemmas that we will use there.
Sections~\ref{sec:abelian} and~\ref{sec:tor}
finally contain the proofs of
Theorems~\ref{thm:intro_abelian} (and its corollaries)
respectively \ref{thm:intro_E_tor}.

\section{Definitions and preliminaries}
\label{sec:prelim}

\noindent 
In this section we fix some definitions and collect a few auxiliary results that are mostly well-known to experts.

Throughout this paper
let $K$ be a field of characteristic zero.
We denote by $\bar{K}$ an algebraic closure of $K$,
and by ${\rm Gal}(K)={\rm Gal}(\bar{K}/K)$ the absolute Galois group of $K$.

\begin{definition}\label{def:var}
By a {\em $K$-variety}
we mean a separated integral scheme of finite type over $K$.
If $X$ is a $K$-variety
and $L/K$ a field extension
we denote by $X_L=X\times_{{\rm Spec}(K)}{\rm Spec}(L)$ the base change,
and if $\pi\colon Y\rightarrow X$ is a morphism of $K$-varieties we write accordingly $\pi_L\colon Y_L\rightarrow X_L$.
If $X$ is a %\footnote{\color{blue}Do we really want the assumption normal here?} 
$K$-variety and $F$ a finite extension of its function field $K(X)$, then we
denote by $g\colon X^{(F)}\to X$ the {\em normalization} of $X$ in $F$, see \cite[II.6.3]{EGAII}. 

A {\em cover}
is a finite surjective morphism $\pi\colon Y\rightarrow X$
of normal $K$-varieties. 
The {\em degree} $\deg(Y/X)$ of a cover $Y\to X$ is the degree of the 
associated function field extension $K(Y)/K(X)$. 
We say that a cover $Y\to X$ of $K$-varieties is {\em geometrically integral}
if $Y$ is geometrically integral as a $K$-variety.
An {\em intermediate cover} of a cover $\pi\colon Y\to X$ is a triple $(Z, f, g)$ where
$Z$ is a normal $K$-variety, $f\colon Y\to Z$ and $g\colon Z\to X$ are covers and $g\circ f=\pi$. 
We then also call $g$ a {\em subcover} of $\pi$. 
Two intermediate covers $(Z, f, g)$ and $(Z', f', g')$ of $\pi$ are said to be isomorphic if there exists an isomorphism $h\colon Z\to Z'$ such that $h\circ f=f'$ and $g'\circ h=g$.

For a cover $\pi\colon Y\rightarrow X$ of $K$-varieties
we define $\mathrm{Ram}(\pi)$ to be the set of all $y\in Y$ such that $\pi$ is ramified at $y$ and call $\mathrm{Ram}(\pi)$ the {\em ramification locus of $\pi$}. The set $\mathrm{Branch}(\pi):=\pi(\mathrm{Ram}(\pi))$ is  the {\em branch locus of $\pi$}.  
The cover $\pi$ is {\em Galois} if the associated function field extension $K(Y)/K(X)$ is a Galois extension.
In that case we denote by
${\rm Gal}(Y/X) := {\rm Gal}(K(Y)/K(X))$ the Galois group of $\pi$,
and 
for $y\in Y$ and $x=\pi(y)$ we let
$$
 D(y/x)=\left\{\sigma\in{\rm Gal}(Y/X)=({\rm Aut}_XY)^{\rm opp}:\sigma(y)=y\right\}
$$
denote the {\em decomposition group} of $y$ over $x$.
\end{definition}

\begin{remark} \label{rem:varieties} 
\label{lem:normalization}Let $X$ be a $K$-variety. 
Then $X$ is geometrically integral if and only if $K$ is algebraically closed in $K(X)$
(cf.~\cite[IV.4.5.9, IV.4.6.1, IV.4.3.1]{EGAIV2}). 
If $X$ is normal, then $X$ is automatically geometrically normal
(cf.~\cite[IV.6.14.2]{EGAIV4}). 
Furthermore, if a separated $K$-scheme of finite type is normal and connected, then it is a $K$-variety. 
If $X$ is a normal $K$-variety and $F$ a finite extension of $K(X)$,
then $g\colon X^{(F)}\to X$ is a cover (finiteness of $g$ follows from \cite[Ch 0, 23.1.1., 23.1.2]{EGAIV1}).
\end{remark}

\begin{remark}\label{rem:closed}
Let $X$ be a $K$-variety, $C\subseteq X$ a closed subset, $K'/K$ a field extension and $u\colon X_{K'}\to X$ the canonical morphism. 
Then $C'=\{x'\in X_{K'}|u(x')\in C\}$ 
is a closed subset of $X_{K'}$, and we endow, as usual, $C$ and $C'$ with the reduced induced 
subscheme structure. The closed immersion $C_{K'}\to X_{K'}$ 
%(\cite[I.4.3.1]{EGAI}) 
has underlying space $C'$ 
%(cf. \cite[I.3.4.8]{EGAI}) 
and $C_{K'}$ is reduced (because $\mathrm{char}(K)=0$, cf. \cite[IV.4.6.1]{EGAIV2}). Thus $C'$ and $C_{K'}$ are $X_{K'}$-isomorphic 
%(\cite[I.5.2]{EGAI}) 
and we will tacitly identify them in the sequel. 
\label{rem:bchange}
Moreover, if  $f\colon X\to Y$ is a cover of $K$-varieties, % and $K'$ an extension field of $K$. Let $C$ be a closed subset of $X$. Then, with the conventions of Remark \ref{rem:closed} in force, 
then
$f_{K'}(C_{K'})=f(C)_{K'}$. % by \cite[I.3.4.8]{EGAI}. 
\end{remark}

\begin{lemma}\label{lem:Branch}
\label{lem:unramet}
\label{lem:ramparts} 
Let $\pi\colon Y\rightarrow X$ be a cover of $K$-varieties.
\begin{enumerate}[(a)]
\item $\pi$ is unramified at $y\in Y$ if and only if $\pi$ is \'etale at $y$.
 \item The branch locus ${\rm Branch}(\pi)$ is a proper closed subset of $X$.
 \item  If $K'/K$ is a field extension, then ${\rm Branch}(\pi_{K'})={\rm Branch}(\pi)_{K'}$. 
\item If $\psi\colon Z\rightarrow Y$ is another cover, 
then $\pi\circ\psi$ is unramified at 
$z\in Z$ if and only if $\psi$ is unramified at $z$ and $\pi$ is unramified at $\psi(z)$. 
In particular, 
${\rm Branch}(\pi\circ \psi)=\pi({\rm Branch}(\psi))\cup{\rm Branch}(\pi)$. 
 \item If $X$ is regular, then every irreducible component of ${\rm Branch}(\pi)$ has codimension one in $X$.
\end{enumerate}
\end{lemma}

\begin{proof}
(a) follows from 
\cite[IV.18.10.1]{EGAIV4} and \cite[Ch.\ 0, 26.1]{EGAIV1},
or see
\cite[Lemma 2.3]{CDJLZ}.
For (b),
${\rm Ram}(\pi)$ is closed 
\cite[Tag 0C3J]{Stacks},
so as $\pi$ is finite (and hence closed), 
this implies that ${\rm Branch}(\pi)=\pi({\rm Ram}(\pi))$ is closed.
The generic fiber of $\pi$ is \'etale because $\mathrm{char}(K)=0$,
hence ${\rm Ram}(\pi)\subsetneqq Y$ and ${\rm Branch}(\pi)\subsetneqq X$.
Part (c) is immediate from \cite[IV.17.7.4]{EGAIV4} and Remark~\ref{rem:closed}.
For (d),
composition of unramified morphisms is unramified \cite[Tag 02G9]{Stacks}.
Conversely,
if $\pi\circ\psi$ is unramified at $z$,
then $\mfm_{X,x}\OOO_{Z,z}=\mfm_{Z,z}$, $\mfm_{X,x}\OOO_{Y,y}\subseteq \mfm_{Y,y}$, and
$\mfm_{Y,y}\OOO_{Z,z}\subseteq \mfm_{Z,z}$,
thus 
$\mfm_{Y,y}\OOO_{Z,z}=\mfm_{Z,z}$,
which shows that $\psi$ is unramified at $z$. 
So by (a), $\psi$ is \'etale at $z$, in particular flat at $z$,
hence \cite[IV.17.7.7]{EGAIV4} gives that $\pi$ is unramified at $y$.
If $X$ is regular as in (e),
then by the purity theorem of Zariski--Nagata \cite[Exp. X, 3.1]{SGA1}, every irreducible component of ${\rm Ram}(\pi)$ 
has codimension one in $Y$,
hence since $\pi$ is finite, ${\rm Branch}(\pi)$ is a finite union of irreducible subsets of $X$ of codimension one.
\end{proof}

\begin{definition}
    Let $K_1,K_2$ be fields of characteristic zero and $g\colon K_1\to K_2$ an isomorphism. 
    Then, pulling back along $\Spec(g^{-1})\colon \Spec(K_1)\to \Spec(K_2)$ induces a functor from the category of $K_2$-varieties to the category of $K_1$-varieties. 
    We denote it  by 
    $$
     Y\mapsto Y^g:=Y\times_{\Spec(K_2), \Spec(g^{-1})} \Spec(K_1)\ \quad\mbox{and}\quad \pi\mapsto \pi^g:=\pi\times_{\Spec(K_2), \Spec(g^{-1})} \Spec(K_1)
    $$ 
    where $\pi\colon Y\to X$ and $\pi^g\colon Y^g\to X^g$. 
 For a $K_2$-variety $X$ we denote by $g_X\colon X\to X^g$ the
    isomorphism
    $X=X^g\times_{{\rm Spec}(K_1),{\rm Spec}(g)}{\rm Spec}(K_2)\rightarrow X^g$.
\end{definition}

\begin{remark}\label{rem:gX}
    \label{p_pi-g}
    Let $K\subseteq K_1,K_2$ be fields and $g\colon K_1\to K_2$ a $K$-isomorphism.
Then $g_X$ is a natural transformation in $X$, i.e.\ for every morphism of $K_2$-varieties $\pi\colon Y\to X$ the following diagram commutes:
        \[
           \begin{tikzcd}
            Y \arrow[r, "g_Y"] \arrow[d, "\pi"] & Y^g \arrow[d, "\pi^g"] \\
            X \arrow[r, "g_X"] \arrow[d]        & X^g \arrow[d]            \\
            {\rm Spec}(K_2) \arrow[r, "{\rm Spec}(g)"]        & {\rm Spec}(K_1).     
    \end{tikzcd}
    \]
        If $\pi\colon Y\to X$ is a Galois cover of  $K_2$-varieties, then $\pi^g\colon Y^g\to X^g$ is a Galois cover of $K_1$-varieties and we have a natural isomorphism $\iota_g\colon \Gal(Y/X)\to \Gal(Y^g/X^g)$
        given by $\sigma\mapsto g_Y\circ\sigma\circ g_Y^{-1}$.
\end{remark}

\begin{remark}
\label{rem:normalization}
\label{rem:GT}
Let $\pi\colon Y\to X$ be a cover of $K$-varieties. 
There is a 1-to-1 correspondence between intermediate fields
of $K(Y)/K(X)$ and isomorphism classes of intermediate covers of $\pi$ given by 
$F\mapsto X^{(F)}$.
If $\pi$ is Galois with Galois group $\Gamma:=\Gal(Y/X)$, 
the canonical map $\Gamma^{\rm opp}\to \Aut_X(Y)$ is bijective and  %$(\pi_*\mathcal{O}_Y)^\Gamma=\mathcal{O}_X$, hence 
the canonical morphism $Y/\Gamma\to X$ is an isomorphism, where $Y/\Gamma$ denotes the geometric quotient of $Y$ by $\Gamma$ in the sense of \cite{SGA1},
so by Galois theory there is a 1-to-1 correspondence between subgroups
of $\Gamma$ and isomorphism classes of intermediate covers of $\pi$ given by 
$H\mapsto X^{(K(Y)^H)}=Y/H$,
where $K(Y)^H$ denotes the fixed field of $H$ in $K(Y)$.
\end{remark}

%\begin{proof}
%Note that $X$ and $Y$ are normal because $\pi$ is a cover (cf.~Definition \ref{def:cover}).
%Hence the claim follows from
%\cite[IV.18.10.1]{EGAIV4} and \cite[Ch. 0, 26.1]{EGAIV1}.
%\end{proof}

\begin{lemma}\label{lem:decomposition_group}
Let $\pi\colon Y\rightarrow X$ be a Galois cover of $K$-varieties
that is unramified at $y\in Y$. 
Let 
$Y\stackrel\rho\rightarrow Z\rightarrow X$ 
be an intermediate cover of $\pi$,
and let $z=\rho(y)$, $x=\pi(y)$.
\begin{enumerate}[(a)]
    \item $K(y)/K(x)$ is Galois
      and there is a canonical isomorphism $D(y/x)\stackrel{\cong}{\rightarrow}{\rm Gal}(K(y)/K(x))$.
    \item ${\rm Gal}(Y/Z)\subseteq D(y/x)$ if and only if $\rho^{-1}(z)=\{y\}$.
    \item $D(y/z)=D(y/x)\cap{\rm Gal}(Y/Z)$
    \item If $Z\rightarrow X$ is Galois, then $D(z/x)$ is the image of $D(y/x)$ under the restriction map ${\rm Gal}(Y/X)\rightarrow{\rm Gal}(Z/X)$, and the following diagram commutes:
    $$
     \xymatrix{
     {\rm Gal}(Y/X)\ar[d]^{\rm res} && D(y/x)\ar[d]^{\rm res}\ar@{_{(}->}[ll]\ar[rr]^\cong && {\rm Gal}(K(y)/K(x))\ar[d]^{\rm res} \\
      {\rm Gal}(Z/X) && D(z/x)\ar@{_{(}->}[ll]\ar[rr]^\cong && {\rm Gal}(K(z)/K(x))
     }
    $$
\end{enumerate}
\end{lemma}

\begin{proof} 
Let $B={\rm Branch}(\pi)$ and let $X'\subseteq X\setminus B$ be an open affine neighbourhood of $x$. 
Then 
$Y'=\pi^{-1}(X')$ is open and affine,
so $Y'=\Spec(R)$ for some normal integral $K$-algebra $R$,
and
$\pi|_{Y'}\colon Y'\to X'$ is \'etale by Lemma \ref{lem:unramet}(a).
The Galois group $\Gamma:=\Gal(Y/X)$ acts on $R$, and  $X'=\Spec(R^\Gamma)$ by Remark \ref{rem:GT}. 
By \cite[Ch. V Thm.\ 2(ii)]{Bou}, the extension $K(y)/K(x)$ is Galois and the canonical homomorphism $\Gamma\to\Gal(K(y)/K(x))$ is surjective. 
Its kernel (the inertia group) is trivial by \cite[Exp V, Cor. 2.4]{SGA1} because $\pi$ is \'etale over $X'$. Thus (a) holds true.
Part (b) follows from \cite[Ch. V Prop.\ 4(i)]{Bou},
and (c), (d) from \cite[Ch. V Prop.\ 7(i)]{Bou}.
\end{proof}

\begin{definition}
Let $X$ be a geometrically integral normal $K$-variety and $Y$ a normal $K$-variety. A cover $\pi\colon Y\rightarrow X$ is 
{\em fully ramified}
if every 
geometrically integral subcover $Z\rightarrow X$ of $\pi$
with $\deg(Z/X)\ge 2$ is ramified.
\end{definition}

\begin{remark}\label{rem:fullyram}
Note that if $Y$ is geometrically integral itself, then this simply amounts to saying that all nontrivial subcovers of $\pi$ are ramified.
If in addition $X$ is an abelian variety,
this coincides with what is called a (PB) cover in \cite{CZ,CDJLZ}.
Clearly every subcover of a fully ramified cover
is again fully ramified.
\end{remark}

{
\begin{lemma}\label{lem:fiber_products}
For $i=1,2$ let $\pi_i\colon Y_i\rightarrow X$ be
a cover of $K$-varieties. Write $Z=Y_1\times_XY_2$ and $\pi_1'\colon Z\rightarrow Y_2$, $\pi_2'\colon Z\rightarrow Y_1$.
Say that $(*)$ holds if $K(Y_1)$, $K(Y_2)$
are linearly disjoint over $K(X)$.
$$
 \xymatrix{
  Y_2\ar[d]_{\pi_2} & Z\ar[l]_{\pi_1'}\ar[d]^{\pi_2'} \\
  X & Y_1\ar[l]^{\pi_1}
 }
$$
\begin{enumerate}[(a)]
    \item If $\pi_1$ is \'etale, then $Z$ is normal and $\pi_1'$ is \'etale.
    %\item If $(*)$ holds, then $Z$ is integral.
    \item If $\pi_1$ is \'etale and $(*)$ holds, then
    $\pi_2'$ is a cover and
    $\pi_1'$ is an \'etale cover.
    \item If $\pi_1$ and $\pi_2$ are \'etale and $(*)$ holds,
     then $\pi_2\circ\pi_1'=\pi_1\circ\pi_2'$ is an \'etale cover.
    \item If $\pi_1$ is \'etale and $\pi_2$ is 
    geometrically integral, fully ramified and Galois,
     then $(*)$ holds.
    \item 
    %{\color{blue} Let $K'$ be the normalization of $K$ in $\Gamma(Y_1, \OOO_{Y_1})$. (Then $K'/K$ is a finite field extension and $Y_1$ a geometrically integral $K'$-variety.) If $\pi_1$ is \'etale
    %and $\pi_2$ is geometrically integral, fully ramified and Galois,
    %then $\pi_2'$
    %is a fully ramified geometrically integral Galois cover of $K'$-varieties.}
    If $\pi_1$ is \'etale
    and $\pi_2$ is geometrically integral, fully ramified and Galois,
    %and there exists an extension $K'/K$ and a $K$-morphism $Y_1\rightarrow{\rm Spec}(K')$ such that
    %$Y_1$ is geometrically integral 
    %as a $K'$-variety,
    then $\pi_2'$
    is a fully ramified geometrically integral Galois cover of $K'$-varieties,
    where $K'$ is the relative algebraic closure of $K$ in $K(Y_1)$.
    
    % \item If 
    % $\pi_1\colon Y_1=X_{K'}\rightarrow X$ is the base change of $X$ to some finite extension $K'/K$, and
    % $\pi_2$ is geometrically integral, fully ramified and Galois,
    % then $\pi_2'$
    % is a geometrically integral fully ramified  Galois cover of $K'$-varieties.
\end{enumerate}
\end{lemma}

\begin{proof}
(a) If $\pi_1$ is \'etale then so is $\pi_1'$, hence $Y_2$ normal implies $Z$ normal.

%(b) If $U=\Spec(R)$ is an affine open subscheme of $X$, then 
%$\pi_i^{-1}(U)=\Spec(R_i)$ is an affine open subset of $Y_i$ and 
%$\pi_1^{-1}(U)\times_U \pi_2^{-1}(U)=\Spec(R_1\otimes_R R_2)$ is an affine open subscheme of $Z$. The natural map $R_1\otimes_R R_2\to K(Y_1)\otimes_{K(X)} K(Y_2)$ is injective and $K(Y_1)\otimes_{K(X)} K(Y_2)$ is a field by $(*)$. Thus $R_1\otimes_R R_2$ is a domain and $\pi_1^{-1}(U)\times_U \pi_2^{-1}(U)$ is integral. It follows that $Z$ is integral because $Z$ is covered by subschemes of the form $\pi_1^{-1}(U)\times_U \pi_2^{-1}(U)$ and any two of them intersect as $X$ is irreducible. 

(b) 
Since $\pi_1,\pi_2$ are finite and surjective, so are $\pi_1',\pi_2'$.
 %The morphisms $\pi_1'$ and $\pi_2'$ are finite and surjective by base change. 
 By (a), $Z$ is normal and $\pi_1'$ \'etale.
 %It remains to show that $Z$ is irreducible. 
 Since the generic fibre of $\pi_i$ is $\Spec(K(Y_i))$, the generic fibre of $\pi_1'$ and of $\pi_2\circ \pi_1'$ is $\Spec(K(Y_1)\otimes_{K(X)} K(Y_2))$, which is irreducible by $(*)$. 
 As $\pi_1'$ is flat, this 
 implies that $Z$ is irreducible (\cite[2.3.5(iii)]{EGAIV2}).

%Let $h:=\pi_2\circ  \pi_1'$ be the morphism $Z\to X$. The restriction $\pi_1'|_{Z_j}: Z_j\to Y_2$ is finite and \'etale. In particular it is surjective. 
%Hence $h|_{Z_j}: Z_j\to X$ is finite
%and surjective. Thus the generic fibre of $h|_{Z_j}$ is $\Spec(K(Z_j))$ and the generic fibre of $h$ is 
%$\coprod_{j=1}^s \Spec(K(Z_j))$. On the other hand the generc fibre of $\pi_i$ is $\Spec(K(Y_i))$. Hence the generic fibre of $h$ is $\Spec(K(Y_1)\otimes_{K(X)} K(Y_2))$. Thus 
%$\Spec(K(Y_1)\otimes_{K(X)} K(Y_2))\cong \coprod_{j=1}^s \Spec(K(Z_j))$. By (*) it follows that $s=1$, as desired. 

(c) follows from (b) since the composition of \'etale covers is an \'etale cover.

(d) Let $F=K(Y_1)\cap K(Y_2)$. Then $X^{(F)}\rightarrow X$ is a subcover of $\pi_2$,
hence geometrically integral and fully ramified (Remark \ref{rem:fullyram}), 
and a subcover of $\pi_1$,
hence \'etale
(Lemma \ref{lem:unramet}(a,d)).
Thus ${\rm deg}(X^{(F)}/X)=1$,
i.e.~$F=K(X)$. 
Since $K(Y_2)/K(X)$ is Galois,
this already implies $(*)$.

(e) %{\color{blue} By construction we have a natural $K$-morphism $Y_1\to \Spec(K')$. Moreover, as $Y_1$ is normal, $K_1$ is the algebraic closure of $K$ in $K(Y_1)$. Thus $Y_1$ is a geometrically integral $K'$-variety by 
%Remark \ref{rem:varieties}.}
Since $Y_1$ is normal, we have a natural $K$-morphism $Y_1\to \Spec(K')$.
We write $\tilde{Y}_1$ and $\tilde{Z}$ when we view $Y_1$ and $Z$ as $K'$-varieties this way,
and we note that $\tilde{Y}_1$ is geometrically integral (Remark \ref{rem:varieties}).
By (d,b), $\pi_2'$ is a Galois cover
with $\Gamma:={\rm Gal}(Z/Y_1)={\rm Gal}(Y_2/X)$.
%By assumption there is a morphism $Y_1\rightarrow{\rm Spec}(K')$,
%and we write $\tilde{Y}_1$ and $\tilde{Z}$ when we view $Y_1$ and $Z$ as $K'$-varieties this way.
For a finite extension $K''/K'$,
the composition $(\tilde{Y}_1)_{K''}\rightarrow Y_1\rightarrow X$
is an \'etale cover of $K$-varieties,
hence (d,b) imply that $(\tilde{Y}_1)_{K''}\times_XY_2=\tilde{Z}_{K''}$ is integral.
Thus $\tilde{Z}$ is geometrically integral.
In view of Remark~\ref{rem:GT} it remains to prove that 
for every proper subgroup $H$ of $\Gamma$ the cover $W':=Y_1^{(K(Z)^H)}\to Y_1$ is ramified.
The cover $W:=X^{(K(Y_2)^H)}\rightarrow X$
is a nontrivial subcover of $\pi_2$, hence fully ramified and ramified,
and since $K(Y_1)$, $K(Y_2)$ are linearly disjoint over $K(X)$, so are $K(Y_1)$, $K(W)$.
Thus $W'=W\times_XY_1\rightarrow W$ is a cover by (b),
which implies that
$W'\rightarrow Y_1$ must be ramified
(Lemma~\ref{lem:ramparts}(d)).
% (g) 
% By (e,c), $\pi_2'$ is a Galois cover
% with $\Gamma:={\rm Gal}(Z/Y_1)={\rm Gal}(Y_2/X)$,
% and $Z=(Y_2)_{K'}$ is a geometrically integral $K'$-variety because $Y_2$ is geometrically integral over $K$.
% In view of Remark~\ref{rem:GT} it remains to prove that 
% for every proper subgroup $H$ of $\Gamma$ the cover $W':=Y_1^{(K(Z)^H)}\to Y_1$ is ramified.
% The cover $W:=X^{(K(Y_2)^H)}\rightarrow X$
% is a nontrivial subcover of $\pi_2$, hence fully ramified and ramified,
% and since $K(Y_1)$, $K(Y_2)$ are linearly disjoint over $K(X)$, so are $K(Y_1)$, $K(W)$.
% Thus $W'=W\times_XY_1\rightarrow W$ is a cover by (c),
% which implies that
% $W'\rightarrow Y_1$ must be ramified
% (Lemma~\ref{lem:ramparts}).
\end{proof}
}

As usual, an abelian variety $A$ over $K$
is a 
proper geometrically integral group scheme over $K$.
We denote by $[n]=[n]_A\colon A\rightarrow A$ the endomorphism $x\mapsto nx$,
and by $A[n]$ its kernel.

\begin{lemma}\label{lem:av_dual_isog}
Let $\alpha\colon A\rightarrow B$ be an isogeny of abelian varieties over $K$ of degree $n$.
There exists an isogeny $\alpha'\colon B\rightarrow A$
with $\alpha'\circ\alpha=[n]_A$ 
and $\alpha\circ\alpha'=[n]_B$. 
\end{lemma}

\begin{proof} See for example \cite[Prop. 5.12]{vdGM}.
%We  identify $B$ with $A/\ker(\alpha)$ via $\alpha$ (cf.\ \cite[p.~118, Corollary~1]%{Mum70}). 
%The finite %\'etale 
%group scheme $\ker(\alpha)$ is killed by $[n]_A$, 
%hence, by \cite[p.~111, Theorem~1]{Mum70}, there exists
%an isogeny $\alpha'\colon B\to A$ such that $\alpha'\circ \alpha=[n]_A$. 
%But then 
%$$
 %\alpha\circ \alpha'\circ \alpha=\alpha\circ [n]_A=[n]_B\circ \alpha,
%$$
%and it follows that $\alpha\circ \alpha'=[n]_B$ because $\alpha$ is surjective and %the schemes $A$ and $B$ are separated and reduced.
\end{proof}

\newcommand{\Sch}{(\mathscr{S}/K)}

For an abelian variety $A$ over $K$ and a point $t\in A(K)$
we denote by
$$
 \tau_t\colon\begin{cases} A&\longrightarrow\; A\\x&\longmapsto\; x+t\end{cases}
$$
the translation by $t$,
which is an automorphism of the variety $A$. Let $\Sch$ be the category of locally noetherian $K$-schemes. 

\newcommand{\Stab}{\mathrm{Stab}}
\newcommand{\uStab}{\underline{\mathrm{Stab}}}
\newcommand{\Hilb}{\mathrm{Hilb}}
\newcommand{\uHilb}{\underline{\mathrm{Hilb}}}
\newcommand{\oK}{\bar{K}}

\begin{lemma}\label{lem_stab} 
Let $A$ be an abelian variety over $K$ and $X$ a nonempty closed subscheme of $A$. 
For a $K$-scheme $T$ we denote $X_T=X\times_{{\rm Spec}(K)}T$.
The functor
$$
\uStab_A(X):\begin{cases} \Sch^{\rm opp}&\longrightarrow\; \mathrm{(groups)},\\ 
T&\longmapsto\; \{a\in A(T): a+X_T=X_T\}\end{cases}
$$
is represented by a smooth closed subgroup scheme $\Stab_A(X)$ of $A$.
\end{lemma}

\begin{proof} 
For an action $*\colon A\times H\to H$ of $A$ on a $K$-scheme $H$ and a $K$-rational point $h\in H(K)$ we define, following 
\cite[\S7 c, page 142]{Mil17}, $\mathrm{Stab}_A(h)$ to be the scheme
theoretic fiber $t_h^{-1}(h)$ of the orbit map 
$t_h\colon A\to H,\ a\mapsto a*h$. Then $\mathrm{Stab}_A(h)$ is a closed subscheme of $A$. From
the universal property of the fiber product we see that it represents
the group valued functor 
$\uStab_A(h)\colon T\mapsto \{a\in A(T): a*h=h\}$
on the category $\Sch^{\rm opp}$,
hence $\Stab_A(h)$ is a subgroup scheme of $A$. 
It is smooth because $\mathrm{char}(K)=0$. 

%$A$ is a projective $K$-scheme by \cite[Thm. 7.1]{Mil86}. Hence, 
Since $A$ is projective,
by \cite[\S3 Thm.\ 3.2 and p.~265]{FGA} (see also \cite[Ch.~5]{FGA2}), the functor\footnote{In detail it is a functor in the following way: For $g\colon T'\to T$ and 
$X\in \uHilb_A(T)$ we have $\uHilb_A(g)(X)=X_{T'}=X\times_T T'$.} 
$$
\uHilb_A:\begin{cases} \Sch^{\rm opp}&\longrightarrow\; \mathrm{(sets)},\\ 
T&\longmapsto\; \{\mbox{closed subschemes of $A_T$ flat over $T$} \},\end{cases}
$$
is represented by a $K$-scheme $\Hilb_A$
which is a disjoint union
of projective $K$-schemes.
%The Hilbert scheme
%$\Hilb_A$ represents the functor\footnote{In detail it is a functor in the following way: For %$g\colon T'\to T$ and 
%$X\in \uHilb_A(T)$ we have $\uHilb_A(g)(X)=X_{T'}=X\times_T T'$.} 
%$$
%\uHilb_A:\begin{cases} (\mathrm{Schemes}/K)^{\rm opp}&\longrightarrow\; \mathrm{(sets)},\\ 
%T&\longmapsto\; \{\mbox{closed subschemes of $A_T$ flat over $T$} \},\end{cases}
%$$
%see \cite[\S3 Thm. 3.2 and p.~265]{FGA}.
The group law $A\times A\to A$ induces an action of $A$ on ${\rm Hilb}_A$, and we view $X$ as  a $K$-rational point:  $X\in {\rm Hilb}_A(K)$. Now ${\rm Stab}_A(X)$ is a smooth closed subgroup scheme of $A$ representing $\uStab_A(X)$.
\end{proof}

\begin{remark} 
In the above situation, if $X$ is in addition reduced, then
        $$
        \Stab_A(X)(\oK)=\{a\in  A(\oK): a+X_{\oK}=X_{\oK}\}
        =\{a\in  A(\oK): a+X(\oK)=X(\oK)\}
        $$
    because $X_{\oK}$ and $a+X_{\oK}$ are reduced closed subschemes
    of $A_{\oK}$.
If moreover $\Stab_A(X)=A$, then $X=A$, since for any fixed $x\in X(\oK)$ we get 
 $A(\bar{K})=A(\bar{K})+x\subseteq X(\bar{K})$.
\end{remark}

\section{Strategy of proof}
\label{sec:sketch}

\noindent
In this section we give a very rough sketch
of the common ideas underlying the proofs of
Theorem \ref{thm:intro_abelian}
in Section \ref{sec:abelian}
and 
Theorem \ref{thm:intro_E_tor}
in Section \ref{sec:tor}.
However, while we hope that this sketch will be helpful to the reader,
the following two sections will use only the numbered definitions and lemmas.

In both proofs, we will first apply
the following technical reduction to fully ramified geometrically integral Galois covers
(where $L$ is the abelian extension from Theorem \ref{thm:intro_abelian}
respectively $L=\mathbb{Q}(A_{\rm tor})$):

\begin{lemma}\label{lem:red_to_irred}
Let 
$K_0$ be a field of characteristic zero,
$A$ be an abelian variety over $K_0$ 
and $L/K_0$ a Galois extension.
Assume that for every finite subextension $K/K_0$ of $L/K_0$,
every nonempty open $U\subseteq A_K$,
every finite Galois extension $K'/K$
linearly disjoint from $L/K$,
every abelian variety $A'$ over $K'$
with an isogeny $\alpha\colon A'\rightarrow A_{K'}$
and every finite collection $(\rho_i\colon Z_i\rightarrow A')_{i=1}^m$
of fully ramified geometrically integral Galois covers
there exists
a finite Galois extension $K''/K$
containing $K'$, and
$$
 x'\in A'(K''L)\cap\alpha^{-1}(U(L))
$$
with
$(\rho_i)_{K''L}^{-1}(x')$ irreducible for every $i$.
Then $A_L$ has WHP.
$$
 \xymatrix{
      & L\ar@{-}[d]\ar@{-}[r] & K'L\ar@{-}[r]\ar@{-}[d]    & K''L\ar@{-}[d] \\
  K_0\ar@{-}[r] & K\ar@{-}[r] & K'\ar@{-}[r] & K'' 
 }
$$
\end{lemma}

\begin{proof}
Let $(\pi_i\colon Y_i\rightarrow A_L)_{i=1}^n$
be a finite collection of ramified covers,
and let $U\subseteq A_L$ be a nonempty open subset.
We have to show that there exists $x\in U(L)\setminus\bigcup_{i=1}^n\pi_i(Y_i(L))$.
We can choose a finite subextension $K/K_0$ of $L/K_0$ such that $U$ and each $\pi_i$ descends to $K$,
so for notational simplicity let us assume that instead
we are given ramified covers $(\pi_i\colon Y_i\rightarrow A_K)_{i=1}^n$
with $(Y_i)_L$ irreducible,
and a nonempty open subset $U\subseteq A_K$.
By shrinking $U$ if necessary, we can assume that each $\pi_i$ is \'etale over $U$ (Lemma \ref{lem:Branch}(a,b)).

Let $F'/K(A)$ be a finite Galois extension containing $K(Y_i)$ for every $i$,
and let $K'$ be the algebraic closure of $K$ in $F'$.
%If $Z$ denotes the normalization of $A_K$ in $F'$, 
Then the normalization
$$
 Z:=A_K^{(F')}\rightarrow A_K
$$ 
is a Galois cover
that has each $\pi_i$ as a subcover.
Let $\varepsilon\colon A'\rightarrow A_K$ be a\footnote{In fact there exists a (up to isomorphism) unique largest unramified subcover, in particular they all factor through $A_{K'}\rightarrow A_K$, but this is not needed here.} maximal unramified subcover of $Z\rightarrow A_K$ factoring through $A_{K'}\rightarrow A_K$.
Then $\rho\colon Z\rightarrow A'$ is a fully ramified, geometrically integral Galois
cover of $K'$-varieties,
and by Lemma \ref{lem:fiber_products}(e) these properties are preserved under extending the base field $K'$,
so in particular $A'$ remains a maximal unramified subcover of $Z\rightarrow A_K$
after replacing
$K'$ by a finite extension $K^\ast$ and accordingly
$F'$ by $F'K^\ast$, $Z$ by $Z_{K^\ast}=A_K^{(F'K^\ast)}$,
and $A'$ by $(A')_{K^\ast}$.
Thus, by enlarging $K'$ by a finite extension if necessary,
we can assume without loss of generality that 
$(\varepsilon^{-1}(0_A))(K')\neq\emptyset$.
%$A'(K')\neq\emptyset$.
Then $A'$ can be given the structure of an abelian variety over $K'$
(this is known as the Lang--Serre theorem,
see e.g.~\cite[10.36]{vdGM})
and we can assume that $\varepsilon$ 
factors through an isogeny $\alpha\colon A'\rightarrow A_{K'}$.
By possibly enlarging $K'$ further,
we may also assume that 
$K'/K$ is Galois and
$\Lambda:={\rm Ker}(\alpha)\subseteq A'(K')$.
Replace $K$ by $L\cap K'$ so that $K'$
is linearly disjoint from $L$ over $K$,
and let $L'=K'L$.
By shrinking $U$ further, we can assume that the cover $Z\rightarrow A_K$
is unramified over $U$.
%Denote by $\rho$ the cover $Z\rightarrow A'$,
%and 
For $\lambda\in\Lambda$ let $\rho_\lambda:=\tau_\lambda\circ\rho$,
where $\tau_\lambda$ denotes translation by $\lambda$.
The following diagram commutes:
\[ 
\xymatrix{
       &&  && Z\ar[lld]\ar[d]^\rho\ar[dr]^{\rho_\lambda} \\
       && Y_i\ar[d]_{\pi_i}  && A'\ar[d]^\alpha\ar_\varepsilon[dll]&A'\ar[l]^{\tau_{-\lambda}}\ar[dl]^\alpha\\
     A && A_K\ar[ll] && A_{K'}\ar[ll]
} 
\]
Since $\rho$  is a fully ramified geometrically integral Galois cover of $K'$-varieties,
so is each $\rho_\lambda$,
hence by assumption there exists 
a finite Galois extension $K''/K$ containing $K'$, and a point
$x'\in A'(K''L)\cap\alpha^{-1}(U(L))$
with $(\rho_\lambda)_{K''L'}^{-1}(x')$ irreducible for every $\lambda\in\Lambda$.
Replacing $K'$ by $K''$ and accordingly
$K$ by $L\cap K''$, $L'$ by $K''L'$,
$A'$ by $A'_{K''}$ and $Z$ by $Z_{K''}$,
we can assume without loss of generality that $K'=K''$.

Let $x:=\alpha(x')\in U(L)$ and suppose
that there exist $i$ and $y\in Y_i(L)$
with $x=\pi_i(y)$.
Let $z\in Z(\bar{K})$ that maps to $y$ under
the cover $Z\rightarrow Y_i$,
and let $x''=\rho(z)$.
Then $\alpha(x'')=x$, hence $x'':=x'-\lambda$ for some $\lambda\in\Lambda$.
Therefore $\rho^{-1}_{L'}(x'')=(\rho_\lambda)^{-1}_{L'}(x')$ is irreducible, i.e.~$\rho^{-1}_{L'}(x'')=\{z\}$.

Base changing everything to $L$ gives the following situation:
$$
 \xymatrix{
  (Y_i)_L\ar[d]_{(\pi_i)_L} & Z_{L'}\ar[l]\ar[d] && y\ar@{|->}[d] & z\ar@{|->}[l]\ar@{|->}[d]\\
  A_L & A'_{L'}\ar[l]&& x & x''\ar@{|->}[l]\\
 }
$$
Trivially $D(z/x)\supseteq D(z/y)$,
so since
$$
 D(z/x)\cong{\rm Gal}(L'(z)/L(x))={\rm Gal}(L'(z)/L)={\rm Gal}(L'(z)/L(y))\cong D(z/y)
$$
by Lemma~\ref{lem:decomposition_group}(a),
we have that $D(z/x)=D(z/y)\subseteq{\rm Gal}(Z_{L'}/(Y_i)_L)$.
Since $\rho_{L'}^{-1}(x'')=\{z\}$,
Lemma~\ref{lem:decomposition_group}(b) implies that ${\rm Gal}(Z_{L'}/A'_{L'})\subseteq D(z/x)$.
Thus ${\rm Gal}(Z_{L'}/A'_{L'})\subseteq{\rm Gal}(Z_{L'}/(Y_i)_L)$,
hence $(\pi_i)_L$ is a subcover of $A'_{L'}\rightarrow A_L$,
contradicting that fact that the former is ramified while the latter is not.
\end{proof}

For this rough sketch we now
assume for simplicity that $K'=K$, $A'=A$, $U=A$ and $m=1$,
i.e.~we are given an abelian variety $A$ over $K$ and a fully ramified geometrically integral Galois cover $\pi\colon Y\rightarrow X:=A$ and want to find $x\in X(L)$ with $\pi_L^{-1}(x)$ irreducible.
For this we will need some assumption on $X(L)$,
so let us assume for simplicity that $X(K)$ is Zariski-dense in $X$.
Let us first look at Haran's proof of his diamond theorem \cite{Haran}, which achieves something similar in the case 
where $X$ is not an abelian variety but $X=\mathbb{P}^1$,
and $\pi$ is merely a Galois cover.
Drastically simplified, and rephrased in geometric terms, 
Haran chooses a suitable finite Galois subextension $M/K$ of $L/K$
and takes the restriction of scalars 
$$
 {\rm Res}_{M/K}(\pi_M)\colon{\rm Res}_{M/K}(Y_M)\longrightarrow{\rm Res}_{M/K}(\mathbb{P}^1_M).
$$
Recall that for an $M$-variety $W$, the restriction ${\rm Res}_{M/K}(W)$ 
is a $K$-variety that represents the functor 
on $K$-schemes
given by $S\mapsto W(S\times_{{\rm Spec}(K)}{\rm Spec}(M))$,
%on field extensions $F/K$ given by $F\mapsto W(M\otimes_KF)$,
and that ${\rm dim}({\rm Res}_{M/K}(W))=n\cdot {\rm dim}(W)$, where $n=[M:K]$.
If we set $G={\rm Gal}(M/K)$ and $\Gamma={\rm Gal}(Y/X)$, then 
the composition 
$$
 \rho\colon {\rm Res}_{M/K}(Y_M)_M\longrightarrow {\rm Res}_{M/K}(\mathbb{P}^1_M)_M\longrightarrow {\rm Res}_{M/K}(\mathbb{P}^1_M)
$$
is again a Galois cover,
with Galois group
the wreath product
$\Gamma\wr G$:

\begin{definition}\label{def:wr}
For groups $\Gamma$ and $G$, we denote by $\Gamma\wr G=\Gamma^G\rtimes G$
the wreath product,
where $G$ acts on the set of functions
$$
 \Gamma^G := \{f\colon G\rightarrow\Gamma \}
$$
from the right by
$$
 f^h(g) = f(hg),\quad f\in\Gamma^G,\;g,h\in G.
$$
It comes with a canonical projection map 
$$
 \mathrm{pr}\colon\begin{cases} \Gamma\wr G&\rightarrow\quad G\\(f,g)&\mapsto\quad g\end{cases}
$$
and, for each $g\in G$,
an evaluation map 
$$
 e_g\colon\begin{cases}\Gamma^G&\rightarrow\quad\Gamma\\f&\mapsto\quad f(g).\end{cases}
$$ 
\end{definition}

Since the rational variety ${\rm Res}_{M/K}(\mathbb{P}^1_M)$
is known to again have HP if $\mathbb{P}^1_K$ has HP,
Haran obtains $y\in{\rm Res}_{M/K}(\mathbb{P}^1_M)(K)$ for which
the fiber $\rho^{-1}(y)$ is irreducible,
cf.~\cite[Prop.~3.3.1]{Serre},
i.e.~$\rho^{-1}(y)={\rm Spec}(F)$ for a finite Galois extension $F/K$.
The group theoretic interplay between
${\rm Gal}(F/K)\cong\Gamma\wr G$
and ${\rm Gal}(L/K)$ then implies
that if $x\in\mathbb{P}^1(M)$ is the point corresponding to $y\in{\rm Res}_{M/K}(\mathbb{P}^1_M)(K)$,
the fiber $\pi_L^{-1}(x)$ is irreducible.

Back to the case of a fully ramified geometrically integral Galois cover $\pi\colon Y\rightarrow X=A$ of our abelian variety $A$, 
also here we can take the restriction
$$
 {\rm Res}_{M/K}(\pi_M)\colon{\rm Res}_{M/K}(Y_M)\longrightarrow B:={\rm Res}_{M/K}(A_M),
$$
where now $B$ is an abelian variety,
and
$$
 \rho\colon{\rm Res}_{M/K}(Y_M)_M\longrightarrow B_M\longrightarrow B
$$
turns out to be again a fully ramified Galois cover (of $K$-varieties) with Galois group
$\Gamma\wr G$,
but the problem is that 
$B$ does in general not have WHP even if $A$ does,
simply since $B(K)$ need not be Zariski-dense in $B$,
even if $A(K)$ was Zariski-dense in $A$.
In fact, if $A(K)=A(M)$, then the Zariski closure of $B(K)$ in $B$ is just
the image of the diagonal embedding
$$
 \Delta\colon A\longrightarrow B={\rm Res}_{M/K}(A_M).
$$
However, we can now restrict the cover $\rho$ to 
the subvariety $\Delta(A)$ of $B$,
similar to what is done in
\cite[Prop.~3.2.1]{Serre} and
\cite[Thm.~3.9, Lemma 3.14]{CDJLZ}.
Let $Z=\Delta^{-1}({\rm Res}_{M/K}(Y_M)_M)$ be the fiber product of $A$ and ${\rm Res}_{M/K}(Y_M)_M$ over $B$. 
$$
 \xymatrix{
  Z\ar[rr]\ar[d] && {\rm Res}_{M/K}(Y_M)_M\ar[d]^{\rho}\\
  A\ar[rr]^\Delta && B
 }
$$
We could then continue similar to Haran,
using the fact that $A$ has WHP by Theorem \ref{thm:CZ},
if only $Z$ was irreducible, but unfortunately it is not.
%If $Z$ was irreducible and normal, then $Z\rightarrow A$ was a fully ramified Galois cover with Galois group $\Gamma\wr G$, and we could continue similar to Haran, 
%using the fact that $A$ has WHP by Theorem \ref{thm:CZ}.
%The main problem now is that $Z$ will not be irreducible in general.
To remedy this, 
we would like to replace $\Delta$ by a different embedding of $A$ into $B$ similar to \cite[Theorem 5.1]{Zannier}.
Equivalently, 
we can keep $\Delta$ and modify $\rho$.
So we first replace the cover $\pi_M$
by 
\begin{equation}\label{eqn:pi1}
 \pi_1:=\tau_t\circ\pi_M\colon Y_1=Y_M\longrightarrow A_M
\end{equation}
where $\tau_t$
is the translation by a suitably chosen point $t\in A(M)$,
and then take ${\rm Res}_{M/K}(\pi_1)$.
This way we can achieve that 
the Galois group of the normalization of an irreducible component of
$Z_t=\Delta^{-1}({\rm Res}_{M/K}(Y_1)_M)$ over $A$
is the full wreath product $\Gamma\wr G$, or at least a sufficiently big subgroup thereof.
More concretely,
since $B_M\cong (A_M)^n$ and 
$$
 {\rm Res}_{M/K}(Y_1)_M\cong\prod_{g\in G}(Y_1)^g=\prod_{g\in G}Y_g
$$ 
where 
$$
 \pi_g:=\pi_1^g=\tau_{g(t)}\circ\pi_M\colon Y_g=(Y_1)^g\longrightarrow A_M
$$ 
is the conjugate cover,
we have that $Z_t$ is the fiber product of the covers $(\pi_g)_{g\in G}$,
and the normalization of an irreducible component of $Z_t$
is the normalization $X^{(F)}$ of $X=A$ in the field compositum $F=\prod_{g\in G}M(Y_g)$.
The point $t$ is then chosen to make the branch loci of these conjugate covers sufficiently distinct,
so that we can apply the following two lemmas:

\begin{lemma}\label{lem:lin_disj_fully_ram}
Let $X$ be a regular geometrically integral $K$-variety,
and for $i=1,\dots,n$ let 
$\pi_i\colon Y_i\rightarrow X$ be a fully ramified geometrically integral Galois cover with Galois group $\Gamma_i$
and branch locus $B_i$.
Assume that for each $1\leq i<j\leq n$,
$B_i$ and $B_j$ have no common irreducible component.
Let $Z=X^{(F)}$ be the normalization of $X$
in the field compositum $F=\prod_{i=1}^nK(Y_i)$.
Then $Z\rightarrow X$ is a fully ramified geometrically integral Galois cover with Galois group $\prod_{i=1}^n\Gamma_i$
and branch locus $\bigcup_{i=1}^nB_i$.
\end{lemma}

\begin{proof}
By induction on $n$ it suffices to consider the case $n=2$.
We claim that $F_0:=K(Y_1)\cap K(Y_2)=K(X)$.
Indeed, $\rho\colon X^{(F_0)}\rightarrow X$
is a subcover both of $Y_1\rightarrow X$ and of $Y_2\rightarrow X$, hence the branch locus $B_0$ of $\rho$
is contained in $B_1\cap B_2$ (Lemma~\ref{lem:ramparts}(d)),
which has codimension at least $2$ in $X$ by assumption,
while every irreducible component of $B_0$ has
codimension 1 in $X$ by the purity theorem (Lemma~\ref{lem:Branch}(e)).
Thus, $\rho$ is unramified, which since $\pi_1$ is fully ramified
and $Y_1$ and therefore $X^{(F_0)}$ are geometrically integral,
implies that ${\rm deg}(\rho)=1$, hence $F_0=K(X)$.

Thus, by Galois theory, ${\rm Gal}(F/K(X))=\Gamma_1\times\Gamma_2$.
We now show that every
(not necessarily geometrically integral) unramified subcover $Y\rightarrow X$ 
of $Z\rightarrow X$ has degree ${\rm deg}(Y/X)=1$.
Since $Y\rightarrow X$ is \'etale (Lemma \ref{lem:unramet}(a)) 
and $Y_1\rightarrow X$ is a fully ramified
geometrically integral Galois cover, 
$Y\times_X Y_1\to Y$ is a cover and 
$Y\times_XY_1\rightarrow Y_1$ is an \'etale cover
(Lemma \ref{lem:fiber_products}(d,b)).
By Galois theory, there exists an intermediate field 
$K(X)\subseteq F_2\subseteq K(Y_2)$ with $F_2K(Y_1)=K(Y\times_XY_1)$,
so $X^{(F_2)}\rightarrow X$ is a subcover both of
$Y\times_XY_1\rightarrow X$, which has branch locus $B_1$ (Lemma \ref{lem:ramparts}(d)),
and of $Y_2\rightarrow X$, which is geometrically integral, fully ramified and has branch locus $B_2$.
By the argument in the first paragraph, we get that ${\rm deg}(X^{(F_2)}/X)=1$,
hence ${\rm deg}(Y\times_XY_1/Y_1)=1$,
which implies that ${\rm deg}(Y/X)=1$.
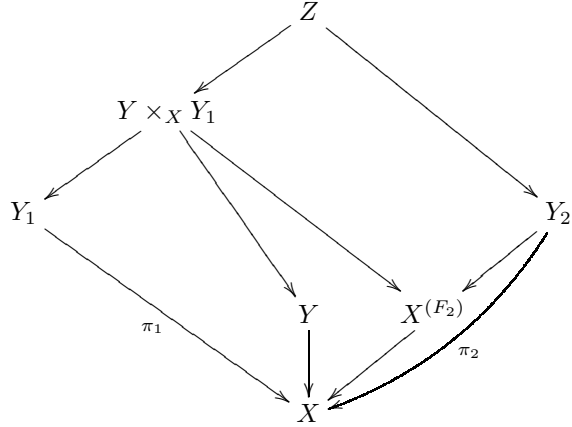
$$
 \xymatrix{
  && Z\ar[dl]\ar[ddrr] &&\\
  & Y\times_XY_1\ar[dl]\ar[ddr]\ar[ddrr] \\
  Y_1\ar[ddrr]_{\pi_1} &&  && Y_2\ar[dl]\ar@/^1.0pc/[ddll]^{\pi_2}\\
   && Y\ar[d] & X^{(F_2)}\ar[dl] \\
  && X 
 }
$$

In particular, $Z$ is a geometrically integral $K$-variety 
and $Z\rightarrow X$ is fully ramified. 
%By Lemma~\ref{lem:fiber_products}(b),
%$Y_1\times_XY_2$ is integral and 
If we set  $U=X\setminus (B_1\cup B_2)$, 
then 
$\pi_1^{-1}(U)\rightarrow U$
and $\pi_2^{-1}(U)\rightarrow U$
are \'etale covers (Lemma~\ref{lem:unramet}(a)).
Moreover,
we have morphisms
$f\colon Y_1\times_XY_2\rightarrow X$ and $g\colon Z\rightarrow Y_1\times_XY_2$.
By Lemma~\ref{lem:fiber_products}(c),
$V:=f^{-1}(U)=\pi_1^{-1}(U)\times_U\pi_2^{-1}(U)\rightarrow U$ is an \'etale cover.
%and so
%$V:=f^{-1}(U)=\pi_1^{-1}(U)\times_U\pi_2^{-1}(U)\rightarrow U$ is an \'etale %cover by
%Lemma~\ref{lem:fiber_products}(d).
In particular it is a subcover of the cover
$(f\circ g)^{-1}(U)=g^{-1}(V)\rightarrow U$,
hence $(f\circ g)^{-1}(U)\rightarrow V$ is an isomorphism,
%Therefore $(f\circ g)^{-1}(U)=g^{-1}(V)\rightarrow V$ is an isomorphism,
and so 
the branch locus $B$ of $Z\rightarrow X$ 
is contained in $X\setminus U=B_1\cup B_2$. On the other hand $B_1\cup B_2\subseteq B$ by Lemma~\ref{lem:unramet}(d), thus $B=B_1\cup B_2$.
%Since $B_i$ for $i=1,2,3$ is again pure of codimension 1 (Lemma~\ref{lem:Branch}(e)),
%it follows that $B_3$ consists of irreducible components of $B_1\cup B_2$.
\end{proof}

\begin{lemma}\label{lem:fiber_product_galois}
Let $X$ be a normal geometrically integral $K$-variety,
$M/K$ a finite Galois extension
with Galois group $G:={\rm Gal}(M/K)$
and
$\pi\colon Y\rightarrow X_M$ a Galois cover with Galois group $\Gamma:={\rm Gal}(Y/X_M)$.
For $g\in G$, we have the Galois cover $\pi^g\colon Y^g\rightarrow X_M$ and the isomorphism $\iota_g\colon \Gamma\to {\rm Gal}(Y^g/X_M)$ (Remark \ref{rem:gX}),
which induce an isomorphism
$$
 \Gamma\wr G=\Gamma^G\rtimes G\cong\left(\prod\nolimits_{g\in G}{\rm Gal}(Y^g/X_M)\right)\rtimes G.
$$
Let 
$F=\prod_{g\in G}M(Y^g)$
be the field compositum.
Then the composition
$$
    X_M^{(F)}\longrightarrow X_M\longrightarrow X
$$ 
is a Galois cover and there is an embedding
\begin{equation*}\label{eqn:phi}
    \psi\colon {\rm Gal}(X_M^{(F)}/X)\hookrightarrow
    \Gamma\wr G
    %\cong
    %\left(\prod\nolimits_{g\in G}{\rm Gal}(Y^g/X_M)\right)\rtimes G
\end{equation*}
such that $\mathrm{pr}\circ \psi$ and $\iota_g\circ e_g\circ \psi|_{{\rm Gal}(X_M^{(F)}/X_M)}$ for $g\in G$
(with $\mathrm{pr}$ and $e_g$ as in Definition \ref{def:wr}) are the restriction maps 
${\rm Gal}(X_M^{(F)}/X)\rightarrow{\rm Gal}(M/K)$
respectively
${\rm Gal}(X_M^{(F)}/X_M)\rightarrow{\rm Gal}(Y^g/X_M)$.
\end{lemma}

\begin{proof}
    Since $X_M^{(F)}\rightarrow X_M$ is a cover (Remark \ref{lem:normalization}),
    so is $X_M^{(F)}\rightarrow X$.
    By Remark \ref{p_pi-g},
    each $g\in G={\rm Gal}(M/K)$ 
    extends to an isomorphism $g_Y\colon M(Y)\rightarrow M(Y^g)$ that is the identity on $K(X)$.
    It follows 
    for $\sigma\in{\rm Gal}(K(X))$
    that
    $\sigma(M(Y))=M(Y^{\sigma|_M})$,
    therefore
    $F$ is invariant under all $\sigma\in{\rm Gal}(K(X))$, hence $X_M^{(F)}\to X$ is Galois.  
    In particular, every $g_Y$ extends further to a $\tilde{g}\in{\rm Gal}(X_M^{(F)}/X)$.
    Let $H := {\rm Gal}(X_M^{(F)}/X)$, $H_1 := {\rm Gal}(X_M^{(F)}/X_M)$ and $H_2 := {\rm Gal}(X_M^{(F)}/Y)$,
    and identify
    $H/H_1 = G$, $H_1/H_2=\Gamma$ via the restriction maps.    
    $$
     
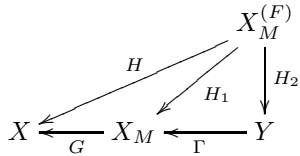
\xymatrix{
        &             & X_M^{(F)}\ar[dll]_H\ar[dl]^{H_1}\ar[d]^{H_2}\\
      X & X_M\ar[l]^G & Y\ar[l]^\Gamma
     }
    $$    
    
    Recall that a permutation group is a pair $(H,T)$ 
    of a group $H$ and a set $T$ together with a faithful action of $H$ on $T$, so that we can identify $H$ with a subgroup
    $H\leq{\rm Sym}(T)$.
    We will view $G$ and $\Gamma$ as permutation groups
    $(G,G)$ and $(\Gamma,\Gamma)$ via their left regular representations.
    Now $H$ acts transitively on $T:=H/H_2$ by left multiplication.
    Since $F=\prod_{g\in G}M(Y^g)$ and $H_2\unlhd H_1$, the kernel of this action is
    $\bigcap_{h\in H}H_2^h\subseteq \bigcap_{g\in G}H_2^{\tilde{g}}=\bigcap_{g\in G}{\rm Gal}(X_M^{(F)}/Y^g)=1$,
    i.e.~the action is faithful,
    and so we obtain a permutation group $(H,T)$.
    
    The set $T_1:=H_1/H_2\in T$ is a block of imprimitivity of $H$,
    and if we set $T_\lambda:=\lambda T_1$ for $\lambda\in\Lambda:= H/H_1=G$, then $(T_\lambda)_{\lambda\in\Lambda}$
    is a block system.
    We thus obtain an action $\theta\colon H\rightarrow{\rm Sym}(\Lambda)$ of $H$ on $\Lambda$ with kernel
    $\bigcap_{h\in H}H_1^h=H_1$ and image
    $G\leq{\rm Sym}(G)={\rm Sym}(\Lambda)$,
    and $\theta\colon H\rightarrow G$ is precisely the restriction map.
    We also have the action $\vartheta\colon H_1\rightarrow{\rm Sym}(T_1)$ of the stabilizer $H_1$ of $T_1$ with kernel
    $\bigcap_{h\in H_1}H_2^h=H_2$ and image $\Gamma\leq{\rm Sym}(\Gamma)={\rm Sym}(T_1)$,
    and $\vartheta\colon H_1\rightarrow\Gamma$
    is precisely the restriction map.
    
    As we view $\Gamma$ and $G$ as permutation groups 
    on $\Gamma=T_1$ respectively $G=\Lambda$, the wreath product $\Gamma\wr G$ comes with a faithful action on $\Gamma\times G=T_1\times\Lambda$ which makes it into the permutational wreath product
    $(\Gamma\wr G,T_1\times\Lambda)$
    \cite[Ch.~I Def.~1.8]{Meldrum}.   
    The embedding theorem for wreath products then gives an embedding of permutation groups 
    $\psi\colon H\rightarrow\Gamma\wr G=\Gamma^G\rtimes G$
    given by
    $$
     \psi(h)=(f,g_0),\quad\mbox{ where }\quad g_0=\theta(h)\quad\mbox{ and }\quad    f(g)=\vartheta(\widetilde{g_0 g}^{-1} h\tilde{g} ),
    $$
    cf.~\cite[Ch.~I Thm.~2.6]{Meldrum} and line (2.8) of its proof\footnote{Note though that the theorem is incorrectly stated, as the stabilizer $H$ of a block $T_\iota$ need not act faithfully on that block, and so the group $H$ in that theorem has to be replaced by its image in ${\rm Sym}(T_{\iota})$. 
    Also the assumption there that the action is imprimitive is unnecessary if one allows the trivial partition as a block system.} or
    \cite[Theorem 8.5]{BMMN} and its proof.
    To conclude, we note that
    ${\rm pr}\circ\psi=\theta$
    by definition,
    and for $h\in H_1$,
    $$
     (\iota_g\circ e_g\circ \psi|_{{\rm Gal}(X_M^{(F)}/X_M)})(h)=
     \iota_g(\vartheta(\tilde{g}^{-1}h\tilde{g}))=\iota_g(g_Y^{-1}\circ h|_{M(Y^g)}\circ g_Y)
      = h|_{M(Y^g)}.\qedhere
    $$
\end{proof}

The proofs of Theorem \ref{thm:intro_abelian}
and
Theorem \ref{thm:intro_E_tor}
differ mainly in how the point $t$ in (\ref{eqn:pi1}) is constructed
(using the assumption on the rank, see Proposition \ref{prop:t}, 
respectively an explicit construction
of points on the elliptic curve using elementary number theory, see Lemma \ref{lem:values_of_pol} and Proposition \ref{prop:FreyJarden})
and how the groups
$\Gamma\wr G$ and ${\rm Gal}(L/K)$ interact
(Lemmas \ref{lem:gt_abelian} respectively \ref{lem:gt_tor}).
For clarity we do not phrase the proofs in terms of restriction of scalars
or fiber products of covers
but instead work directly with 
normalizations in field composita.

\section{Abelian extensions}
\label{sec:abelian}

\noindent
In this section we discuss 
ranks of abelian varieties (in particular the Frey--Jarden conjecture and some consequences)
and then prove Theorem \ref{thm:intro_abelian}.

\begin{definition}
The {\em (rational) rank}
of an abelian group $\Gamma$
is ${\rm rk}(\Gamma)={\rm dim}_\mathbb{Q}(\Gamma\otimes_\mathbb{Z}\mathbb{Q})$.
For a field extension $L/K$ and an abelian variety $A$ over $K$,
we call ${\rm rk}(A(L))$ the rank of $A$ over $L$.
\end{definition}

\begin{remark}\label{lem:rank}
Let $A,B$ be abelian varieties over $K$.
Then ${\rm rk}((A\times B)(K))={\rm rk}(A(K))+{\rm rk}(B(K))$.
If $A$ is isogeneous to $B$,
then ${\rm rk}(A(K))={\rm rk}(B(K))$.
\end{remark}

%\begin{proof}
%Clear.
%\end{proof}

\begin{proposition}[Frey--Jarden]\label{prop:FreyJarden}
Let $E$ be an elliptic curve over $\mathbb{Q}$ with affine model 
$$
 E_0 : Y^2 = f(X), \quad f\in\mathbb{Z}[X].
$$
For $a\in\mathbb{Q}$ 
let $\sqrt{f(a)}$ denote one of the square roots of $f(a)$ in $\bar{\mathbb{Q}}$ and 
let
$$
 x_a := (a,\sqrt{f(a)}) \in E_0(\mathbb{Q}(\sqrt{f(a)})).
$$
There exists a sequence $(p_i)_{i=1}^\infty$ of prime numbers such that the sequence
$(x_{p_i^{-1}})_{i=1}^\infty$ of points is linearly independent in $E(\bar{\mathbb{Q}})$.
In particular,
${\rm rk}(E(\mathbb{Q}^{\rm ab}))=\infty$.
\end{proposition}

\begin{proof}
This is what is actually proven in \cite[Theorem 2.2]{FreyJarden}.
\end{proof}

\begin{question}[Frey--Jarden]\label{q:FreyJarden}
Does every nonzero abelian variety $A$ over $\mathbb{Q}$
have infinite rank over $\mathbb{Q}^{\rm ab}$?
\end{question}

If Question \ref{q:FreyJarden} has a positive answer,
then also every nonzero abelian variety $A$ over $K^{\rm ab}$
has infinite rank over $K^{\rm ab}$,
for every number field $K$:

\begin{lemma}\label{lem:rank_res}
Let $K\subseteq M\subseteq L$ be fields with $M/K$ Galois and $L/M$ algebraic.
If every nonzero abelian variety over $K$
has infinite rank over $M$,
then every nonzero abelian variety over $L$
has infinite rank over $L$.
\end{lemma}

\begin{proof}
We first show that every nonzero abelian variety $A$ over $M$ has infinite rank over $M$.
There exists a finite Galois subextension $K_0/K$ of $M/K$
and an abelian variety $A_0/K_0$ with $(A_0)_{M}=A$.
For $n=[K_0:K]$,
the restriction of scalars $B={\rm Res}_{K_0/K}(A_0)$
is an abelian variety over $K$ 
of dimension $n\cdot{\rm dim}(A)>0$,
see \cite[Section~7.6]{BLR},
hence ${\rm rk}(B(M))=\infty$ 
by assumption.
As $B(M)=A_0(M\otimes_{K}K_0)=A_0(M^n)=A_0(M)^n$,
we conclude (with Remark \ref{lem:rank}) that also ${\rm rk}(A(M))={\rm rk}(A_0(M))=\infty$.

If we now take a nonzero abelian variety $A$ over $L$,
then similarly we find a finite subextension $M_1/M$ of $L/M$ and an abelian variety $A_1/M_1$ with $(A_1)_L=A$, and $C={\rm Res}_{M_1/M}(A_1)$ is an abelian variety over $M$ with
${\rm rk}(A(L))\geq{\rm rk}(A_1(M_1))={\rm rk}(C(M))=\infty$.
\end{proof}

Thus, if %Assuming that 
Question \ref{q:FreyJarden}
has a positive answer, the assumptions of
Proposition~\ref{prop:t} below are therefore satisfied in the case 
where $K$ is a number field
and $L=K^{\rm ab}$.
The point $t$ constructed there will be essential
in our proof of Theorem \ref{thm:intro_abelian}.

\begin{lemma}\label{lem:t1}
Let $M_1,\dots,M_r$ be Galois extensions of $K$ such that
$M_i\not\subseteq M_1\cdots M_{i-1}$ for every $i$.
Then there exists $\sigma\in{\rm Gal}(M_1\cdots M_r/K)$
such that $\sigma|_{M_i}\neq 1$ for every $i$.
\end{lemma}

\begin{proof}
Induction on the stronger statement that
every $\sigma_0\in{\rm Gal}(M_1\cdots M_{r-1}/K)$ 
extends to a $\sigma\in{\rm Gal}(M_1\cdots M_r/K)$
with $\sigma|_{M_r}\neq 1$.
Let $M_0=(M_1\cdots M_{r-1})\cap M_r$.
Note that $M_0/K$ and $M_r/K$ are Galois and $M_0\subsetneqq M_r$.
Therefore
$\sigma_0|_{M_0}$ extends to some $1\neq\sigma_r\in{\rm Gal}(M_r/K)$,
and there exists a unique $\sigma\in{\rm Gal}(M_1\cdots M_r/K)$ with
$\sigma|_{M_1\cdots M_{r-1}}=\sigma_0$
and $\sigma|_{M_r}=\sigma_r$.
\end{proof}

\begin{lemma}\label{lem:t2}
Let $A$ be an abelian variety over $K$ and $x\in A(\bar{K})$.
There exists $m>0$ such that $K(nmx)=K(mx)$ for every $n>0$.
\end{lemma}

\begin{proof}
This holds since $K(x)/K$ is finite
and $K(nmx)\subseteq K(mx)$ for every $m,n$.
\end{proof}

\begin{lemma}\label{lem:t3}
Let $k>0$, $r\geq0$ and let $V_1,\dots,V_r\leq\mathbb{Z}^k$ be subgroups of infinite index.
Then there exists $x\in\mathbb{Z}^k$ such that
$nx\notin\bigcup_{j=1}^rV_j$ for every $n>0$.
\end{lemma}

\begin{proof}
Since 
$V_j\otimes_\mathbb{Z}\mathbb{Q}$ is a proper subspace of the vector space $\mathbb{Q}^k$,
we can pick $y\in\mathbb{Q}^k\setminus \bigcup_{j=1}^r V_j\otimes_\mathbb{Z}\mathbb{Q}$,
and then $\lambda y\notin \bigcup_{j=1}^r V_j\otimes_\mathbb{Z}\mathbb{Q}$ for every $\lambda\in\mathbb{Q}^\times$.
So if $m\neq 0$ is such that $my\in\mathbb{Z}^k$,
then $x:=my$ satisfies the claim.
\end{proof}

\begin{proposition}\label{prop:t}
Let $A$ be an abelian variety over $K$,
and let $L/K$ be an abelian extension.
Let $\alpha\colon\prod_{i=1}^kA_i\rightarrow A$
be an isogeny
where each $A_i$ is a simple abelian variety over $K$ with
${\rm rk}(A_i(L)/A_i(K_1))=\infty$
for every finite subextension $K_1/K$ of $L/K$.
For every finite set $S\subseteq A(\bar{K})$
and proper abelian subvarieties $B_1,\dots,B_r\subsetneqq A$
there exist $t\in A(L)$ and $\sigma\in{\rm Gal}(L/K)$
with 
$$
 t-\sigma(t)\;\notin\; S+\bigcup_{j=1}^rB_j(\bar{K}).
$$ 
\end{proposition}

\begin{proof}
Since ${\rm rk}(A_1(L)/A_1(K))>0$,
there exists $t_1\in A_1(L)$ with ${\rm ord}(t_1+A_1(K))=\infty$
(i.e.~the image of $t_1$ in $A_1(L)/A_1(K)$ is of infinite order).
Let $M_1=K(t_1)$ and assume without loss of generality (by Lemma \ref{lem:t2}) that $K(t_1)=K(nt_1)$ for every $n>0$.
Now similarly there exists $t_2\in A_2(L)$ with ${\rm ord}(t_2+A_2(M_1))=\infty$.
Let $M_2=K(t_2)$ and assume without loss of generality that $K(t_2)=K(nt_2)$ for every $n>0$.
Iterate this to obtain
for every $1\leq i\leq k$
a $t_i\in A_i(L)$ with $M_i:=K(t_i)=K(nt_i)$ for every $n>0$
and
%a finite extension $M_i/K$ and $t_i\in A_i(M_i)$ with $K(t_i)=K(nt_i)$ for every $n>0$ and 
${\rm ord}(t_i+A_i(M_1\cdots M_{i-1}))=\infty$,
in particular $M_i\not\subseteq M_1\cdots M_{i-1}$.
Since $L/K$ is abelian, $M_i/K$ is Galois for every $i$.
By Lemma \ref{lem:t1} there exists $\sigma\in{\rm Gal}(L/K)$
with $\sigma(t_i)\neq t_i$ for every $i$.

Let $\iota_i\colon A_i\rightarrow A_1\times\dots\times A_k$ denote the canonical embedding setting the other components to $0$,
and for $j=1,\dots,r$ let $\pi_j\colon A\rightarrow C_j:=A/B_j$
denote the quotient map.
Any nonzero simple quotient of $C_j$ is again isogeneous to some $A_{i_j}$ and we can choose the $i_j$ such that the composition
$$
 \beta_j\colon A_{i_j} \stackrel{\iota_{i_j}}\longrightarrow A_1\times\dots\times A_k\stackrel{\alpha}{\longrightarrow} A\stackrel{\pi_j}{\longrightarrow} C_j\longrightarrow A_{i_j}
$$
is nonzero, hence surjective (as $A_{i_j}$ is simple), hence an isogeny.
Thus $t_{i_j}':=\beta_j(t_{i_j})$ 
again satisfies $K(nt_{i_j}')=K(t_{i_j})$ for every $n>0$
(due to the isogeny $\beta_j'$ from Lemma \ref{lem:av_dual_isog}
which maps $t_{i_j}'$ to ${\rm deg}(\beta_j)t_{i_j}$), 
and therefore also the image $c_j:=\pi_j(\alpha(\iota_{i_j}(t_{i_j})))$ of $t_{i_j}$ in $C_j$
satisfies $K(n c_j)=K(t_{i_j})=M_{i_j}$ for every $n>0$.
In particular, $\sigma(n c_j)\neq n c_j$ for every $n>0$.

Let $C_j(L)^{\left<\sigma\right>}\leq C_j(L)$ be the subgroup fixed by $\sigma$,
and write $u_{j,i}=\pi_j(\alpha(\iota_i(t_i)))$.
The map 
$$
 \begin{cases}\quad\quad\mathbb{Z}^k&\longrightarrow\quad C_j(L)/C_j(L)^{\left<\sigma\right>},\\
(n_1,\dots,n_k)&\longmapsto\quad\sum_{i=1}^k n_iu_{j,i}+C_j(L)^{\left<\sigma\right>}
\end{cases}
$$
is a homomorphism with infinite image
(since $nu_{j,i_j}=nc_{j}\notin C_j(L)^{\left<\sigma\right>}$ for every $n>0$),
hence its kernel 
$$
 W_j=\left\{(n_1,\dots,n_k)\in\mathbb{Z}^k:\sigma\left(\sum\nolimits_{i=1}^k n_iu_{j,i}\right)=\sum\nolimits_{i=1}^k n_iu_{j,i}\right\}
$$
has infinite index in $\mathbb{Z}^k$.
By Lemma \ref{lem:t3},
there exists $x=(n_1,\dots,n_k)\in\mathbb{Z}^k$ such that 
$nx\notin\bigcup_{j=1}^rW_j$ for every $n>0$.
That is,
if we write
$$
 t_0:=\alpha(n_1t_1,\dots,n_kt_k)=\sum_{i=1}^kn_i\alpha(\iota_i(t_i))\in A(L),
$$
then $\sigma(n\pi_j(t_0))\neq n\pi_j(t_0)$
for every $n>0$ and every $j$.

We claim that for every $j$,
$nt_0-\sigma(nt_0)\in S+B_j(\bar{K})$ for at most finitely many $n$.
Indeed, for any such $n$,
$\pi_j(nt_0-\sigma(nt_0))$ lies in the finite set $\pi_j(S)$
so if there were infinitely many such $n$,
by the box principle there exist $n''>n'>0$ with
$\pi_j(n't_0-\sigma(n't_0))=
\pi_j(n''t_0-\sigma(n''t_0))$,
hence $n:=n''-n'>0$ satisfies
$0=\pi_j(nt_0-\sigma(nt_0))=n\pi_j(t_0)-\sigma(n\pi_j(t_0))$,
contradiction.
Therefore
there exists (cofinitely many) $n_0>0$  such that
$n_0t_0-\sigma(n_0t_0)\notin S+\bigcup_{j=1}^rB_j(\bar{K})$,
i.e.~the point $t:=n_0t_0\in A(L)$ satisfies the 
claim of the proposition.
\end{proof}

\begin{remark}\label{rem:rankcondition}
\begin{enumerate}[(a)]
\item
As indicated above, 
the condition that
${\rm rk}(A_i(L)/A_i(K_1))=\infty$
for every finite subextension $K_1/K$ of $L/K$
is satisfied if $K$ is finitely generated 
and ${\rm rk}(A_i(L))=\infty$,
since in this case also $K_1$ is finitely generated
and therefore $A_i(K_1)$ is finitely generated
by the N\'eron--Mordell--Weil theorem \cite[Ch.~6 Thm.~1]{Lang_fundamentals}.
\item
The condition is also satisfied
if $K$ is ample (see introduction) and $L/K$ is any infinite extension:
Let $L_1$ be a proper finite extension of $K_1$ contained in $L$.
The restriction of scalars $R={\rm Res}_{L_1/K_1}((A_i)_{L_1})$
is an abelian variety over $K_1$
and we have the diagonal embedding
$\Delta\colon (A_i)_{K_1}\rightarrow R$,
hence $R$ is isogeneous to $(A_i)_{K_1}\times B$ for an abelian variety $B$ over $K_1$,
which is nonzero since 
${\rm dim}(R)=[L_1:K_1]\cdot{\rm dim}(A_i)>{\rm dim}(A_i)$,
hence 
$$
 {\rm rk}(A_i(L)/A_i(K_1))\geq {\rm rk}(A_i(L_1)/A_i(K_1)) = {\rm rk}(R(K_1)/\Delta(A_i(K_1))) 
 = {\rm rk}(B(K_1)) = \infty
$$
by \cite[Theorem 1.2]{FehmPetersen}.
\end{enumerate}
\end{remark}

The final ingredient we need for our proof of Theorem \ref{thm:intro_abelian}
is the following purely group theoretic lemma
regarding wreath products (cf.~Definition \ref{def:wr}).
For a group $H$ we denote by 
$H'=[H,H]$ the commutator subgroup.

\begin{lemma}\label{lem:gt_abelian}
Let $\Gamma$ and $G$ be groups
and
%and let $\pi\colon \Gamma^G\rtimes G\rightarrow G$ be the projection. 
let $H\leq \Gamma\wr G$ such that ${\rm pr}(H)=G$ and  $H\cap \Gamma^G$ surjects onto $\Gamma^{\{1,g\}}$ for some $1\neq g\in G$ under the restriction map
%$\Gamma^G \to \Gamma^{\{1,g\}}$.  
$f\mapsto f|_{\{1,g\}}$.
Then for every $H'\leq N\leq \Gamma\wr G$ 
we have that $N \cap \Gamma^G$ surjects onto $\Gamma^{\{1\}}$ under the restriction map
$f\mapsto f|_{\{1\}}$.
\end{lemma}

\begin{proof}
Let $\gamma\in \Gamma$. 
It suffices to find $\nu \in N\cap \Gamma^G$ such that $\nu(1)=\gamma$. 
Since ${\rm pr}(H)=G$,
there exists $h = \rho g^{-1}\in H$ such that ${\rm pr}(h)=g^{-1}$ and $\rho = hg\in \Gamma^G$. 
Since $H\cap\Gamma^G$ surjects onto $\Gamma^{\{1,g\}}$, 
there exists $\eta\in H\cap\Gamma^G$ such that
\[
    \eta(1) = \gamma^{-1} \quad \mbox{and} \quad \eta(g) = 1.
\]
Since $h,\eta\in H$ we have that $\nu:= [h,\eta] \in H'\leq N$,
and since $\eta\in \Gamma^G\unlhd \Gamma\wr G$ we have that
%we have $\nu=[h,\eta]=\eta^{h^{-1}}\eta^{-1}\in H\cap\Gamma^G\leq\Gamma^G$.
%while on the other hand, 
\[
    \nu = [h,\eta] = h\eta h^{-1}\eta^{-1} = \rho \eta^{g} \rho^{-1} \eta^{-1} \in \Gamma^G.
\]
%Putting $\gamma_1:= \rho(1)$, we get that
This finishes the proof since
\[
   \nu(1) =  \rho(1) \eta(g) \rho^{-1}(1) \eta^{-1}(1) = \rho(1)\cdot 1 \cdot\rho(1)^{-1}\cdot \gamma =\gamma. \qedhere
\]
\end{proof}

\begin{proof}[Proof of Theorem \ref{thm:intro_abelian}]
There exists
an isogeny $\beta\colon\prod_{i=1}^kA_i\rightarrow A_L$
with $A_i$ simple abelian varieties over $L$,
see e.g.~\cite[Prop. 12.1]{Mil86} and the text below that proposition.
We let $\beta'\colon A_L\rightarrow\prod_{i=1}^kA_i$
denote the isogeny with $\beta\circ\beta'=[{\rm deg}(\beta)]$
(Lemma \ref{lem:av_dual_isog}).
Replacing $K$ by a finite extension inside $L$
we may assume without loss of generality that $A_i$, $\beta$, $\beta'$ all descend to $K$.
Moreover, since each $A_i$ is a nonzero homomorphic image of $A_L$ (via $\beta'$),
by the assumption ${\rm rk}(A_i(L))=\infty$
we can assume without loss of generality
that ${\rm rk}(A_i(K))>0$ for every $i$,
which implies (via $\beta$) that $A(K)$ is Zariski-dense in $A$.
We also note that the assumption of Proposition
\ref{prop:t} is satisfied for $A$ and the extension $L/K$,
see Remark \ref{rem:rankcondition}(a).

Let
$U\subseteq A$ be nonempty and open, 
let $K'/K$ be a finite Galois extension
linearly disjoint from $L/K$,
let $\alpha\colon A'\rightarrow A_{K'}$ be an isogeny,
and let $(\pi_i\colon Y_i\rightarrow A')_{i=1}^n$
be a finite collection of fully ramified geometrically integral Galois covers.
Let $L':=K'L$, note that $L'/K'$ is again abelian,
and let $X_1',\ldots, X_s'$ be the irreducible components of 
$\bigcup_{i=1}^n{\rm Branch}(\pi_i)$.
By Lemma \ref{lem:red_to_irred}, it suffices to find
a finite Galois extension $K''/K$ containing $K'$ and
\begin{equation}\label{eqn:goal}
x'\in A'(K''L)\cap\alpha^{-1}(U(L))\;
\mbox{ with }\;(\pi_i)_{K''L}^{-1}(x')\mbox{ irreducible for every }i.
\end{equation}
The freedom to choose $K''$
allows us, using Lemma \ref{lem:fiber_products}(e), to freely replace
$K$ by a finite extension $K^*$ of $K$ inside $L$ (and accordingly $K'$ by $K^*K'$),
so 
since ${\rm Branch}((\pi_i)_{K^*K'})={\rm Branch}(\pi_i)_{K^*K'}$ 
by Lemma~\ref{lem:ramparts}(c), 
assume without loss of generality
that each $(X_\mu')_{L'}$ is irreducible.
%a finite Galois extension $K''/K$ containing $K'$, and
%$x'\in A'(K'L)\cap\alpha^{-1}(U(L))$
%with $(\pi_i)_{K'L}^{-1}(x')$ irreducible for every $i$.
%Possibly enlarging $K$ one more time,
%we can assume without loss of generality that
%$K=K'\cap L$.
Let $m={\rm deg}(\alpha)$
and let $\alpha'\colon A_{K'}\rightarrow A'$ be the isogeny
with $\alpha\circ\alpha'=[m]_{A_{K'}}$  (Lemma \ref{lem:av_dual_isog}).

Let 
$u\colon A_{K'}\rightarrow A$ be the base change,
note that ${\rm Branch}(u\circ\alpha\circ\pi_i)=u(\alpha({\rm Branch}(\pi_i)))$ by Lemma~\ref{lem:ramparts}(d),
and let
$X_\mu := u(\alpha(X_\mu'))$.
By Lemma \ref{lem:Branch}(b), each $X_\mu'$ is a proper closed subset of $A'$,
hence $X_\mu$ is a proper closed subset of $A$. 
For each $1\leq\mu,\nu\leq s$,
$$
 S_{\mu,\nu}:=\{a\in A(\bar{K}):a+X_\mu(\bar{K})=X_\nu(\bar{K})\}
$$
is either empty or a coset of the stabilizer ${\rm Stab}_A(X_\nu)(\bar{K})=S_{\nu,\nu}$, 
where ${\rm Stab}_A(X_\nu)$ is a smooth closed subgroup scheme of $A$ (see Lemma \ref{lem_stab}). 
Since $\emptyset\neq X_\nu\subsetneqq A$,
we have that ${\rm Stab}_A(X_\nu)\neq A$. 
Therefore, $\bigcup_{\mu,\nu}S_{\mu,\nu}$
is contained in a set of the form
$S+\bigcup_{j=1}^rB_j(\bar{K})$
with $S\subseteq A(\bar{K})$ finite and
the $B_j$ proper abelian subvarieties of $A$.
Without loss of generality, $0_A\in S$.
Proposition \ref{prop:t} gives $t\in A(L)$ and $\hat{\sigma}\in{\rm Gal}(L/K)$
with 
$$
 t-\hat{\sigma}(t)\notin [m]^{-1}S+\bigcup_{j=1}^rB_j(\bar{K}) = [m]^{-1}\left(S+\bigcup_{j=1}^rB_j(\bar{K})\right),
$$ 
in particular $mt-\hat{\sigma}(mt)\not\in \bigcup_{\mu,\nu}S_{\mu,\nu}$.
Let $M=K(t)$, $M'=K'M$, $\sigma=\hat{\sigma}|_M$,
and note that $M/K$ is a Galois extension (since $L/K$ is abelian)
and that
$$
 \sigma\in G:={\rm Gal}(M/K)={\rm Gal}(M'/K').
$$
$$
 \xymatrix{
  K'\ar@{-}[r]^G\ar@{-}[d] & M'\ar@{-}[r]\ar@{-}[d] & L'\ar@{-}[d]\\
  K\ar@{-}[r]^G & M\ar@{-}[r] & L
 }
$$
Let $t':=\alpha'(t)\in A'(M')$.
For $g\in G$,
$$
 \pi_{i,g}:=\tau_{g(t')}\circ(\pi_i)_{M'}\colon Y_{i,g}=(Y_i)_{M'}\rightarrow A'_{M'}
$$ 
is a fully ramified geometrically integral Galois cover of $M'$-varieties
(Lemma \ref{lem:fiber_products}(e)).
Note that $\pi_{i,g}=(\pi_{i,1})^g$,
which induces a natural isomorphism
$\Gamma_i:={\rm Gal}(Y_{i,1}/A'_{M'})\rightarrow {\rm Gal}(Y_{i,g}/A'_{M'})$
%whose Galois group we can identify with %$\Gamma_i:=\Gal(Y_i/A')$
(Remark~\ref{rem:gX}).
Let $W_i$ be the normalization of $A'_{M'}$
in the field compositum $\prod_{g\in G}M'(Y_{i,g})$.
By Lemma~\ref{lem:fiber_product_galois},
$\delta_i\colon W_i\rightarrow A'_{M'}\rightarrow A'$ is
a Galois cover
whose Galois group 
$H_i$ embeds into $\Gamma_i\wr G$ 
such that,
after identifying $H_i$ with its image in $\Gamma_i\wr G$,
the restriction of the maps $\mathrm{pr}$ and $e_g$ 
to $H_i$ respectively $H_i\cap\Gamma_i^G$
coincide with the restriction maps $H_i\to G$ respectively $\Gal(W_i/A'_{M'})\to \Gal(Y_{i,g}/A'_{M'})=\Gamma_i$.

For $g\in G$, using Lemma \ref{lem:ramparts}(d) we see that
$$
 {\rm Branch}(\pi_{i,g})={\rm Branch}(\tau_{g(t')}\circ(\pi_i)_{M'})={\rm Branch}(\pi_i)_{M'}+g(t')
$$
has irreducible components
among the $(X_\mu')_{M'}+g(t')$, $\mu = 1,\ldots, s$.
Since $K'$ and $M$ are linearly disjoint over $K$,
and $u_M$ is the base change morphism $A_{M'}\rightarrow A_M$,
using Remark \ref{rem:closed} we obtain 
\begin{eqnarray*}
    %\begin{split}
    (u_M\circ\alpha_{M'})\big( (X'_{\mu})_{M'} + g(t')\big) 
    &=& u_M\big( \alpha_{M'}((X'_{\mu})_{M'}) + \alpha_{M'}(g(t'))\big) \\ 
    &=& u_M\big( \alpha_{M'}((X'_{\mu})_{M'}) + g(mt) \big)\\
    &=&u_M\big( \alpha_{M'}((X'_{\mu})_{M'})\big) + u_M(g(mt)) \\ 
    &=& (u(\alpha(X'_{\mu})))_M+ g(mt) \\
    &=& (X_{\mu})_M +g(mt).
%    \end{split}
\end{eqnarray*}
From this, it follows that 
${\rm Branch}( \pi_{i,1})$ and ${\rm Branch}( \pi_{i,\sigma})$
have no common component. Indeed, otherwise 
there exist $1\leq \mu,\nu\leq s$ with
$(X_\mu')_{M'}+t'=(X_\nu')_{M'}+\sigma(t')$, and thus by the above computation, $(X_\mu)_M + mt = (X_{\nu})_M +\sigma(mt)$, and so $mt-\sigma(mt)\in S_{\mu,\nu}$, a contradiction.

Therefore, by Lemma \ref{lem:lin_disj_fully_ram} (with $n=2$)
the normalization of $A'_{M'}$ in the compositum
$R_i:=M'(Y_{i,1})M'(Y_{i,\sigma})$
is a fully ramified geometrically integral Galois cover
$\rho_i\colon Z_i\rightarrow A'_{M'}$ with Galois group
$$
 {\rm Gal}(Z_i/A'_{M'})={\rm Gal}(Y_{i,1}/A'_{M'})\times{\rm Gal}(Y_{i,\sigma}/A'_{M'})=\Gamma_i^{\{1,\sigma\}}\leq\Gamma_i^G.
$$
The inclusions $R_i\subseteq\prod_{g\in G}M'(Y_{i,g})$,
$M'(Y_{i,1})\subseteq R_i$ and $M'(Y_{i,\sigma})\subseteq R_i$
induce covers $W_i\rightarrow Z_i$, $Z_i\rightarrow Y_{i,1}$ and $Z_i\rightarrow Y_{i,\sigma}$
such that the following diagram commutes:
$$
 \xymatrix{
    &&&&&&&W_i\ar[dddllll]^{H_i}_{\delta_i}\ar[d]\\
    &&&&&&& Z_i\ar[dl]\ar[dd]^{\Gamma_i^{\{1,\sigma\}}}_{\rho_i}\ar[dr] \\
    &A_{K'}\ar[d]_{[m]}\ar[drr]^{\alpha'}& &Y_i\ar[d]_{\pi_i}^{\Gamma_i}&&& Y_{i,1}\ar[dr]^{\pi_{i,1}} && Y_{i,\sigma}\ar[dl]_{\pi_{i,\sigma}} \\ 
   A&A_{K'}\ar[l]^u&&A'\ar[ll]^\alpha &&&& A'_{M'}\ar[llll]_G
 }
$$
Let $\Omega:=A(K)$ and
$$
 \Omega':=\alpha'( \Omega )\subseteq A'(K')\subseteq A'({M'}).
$$ 
Since $\Omega$ is a dense subgroup of $A$
and $\alpha'$ is an isogeny, $\Omega'$ is a dense 
subgroup of $A'$.
Thus by %Lemma \ref{lem:WHP_implies_IP_mod}
\cite[Theorem 1.4]{CDJLZ} applied to $A'_{M'}$ and $\Omega'$,
there exists 
a finite index coset $C'\subseteq\Omega'$ such that
for every $c'\in C'$ each of the fibers $\rho_i^{-1}(c')$ is integral.
Then $C:=\Omega\cap\alpha'^{-1}(C')$ is a finite index coset of $\Omega$,
hence Zariski-dense in $A$,
and so we can pick
$c\in C\cap\tau_{t}([m]^{-1}(U))$
and let $c':=\alpha'(c)$.
Without loss of generality, each $\delta_i$ is unramified over $c'$.

We claim that 
$$
 x':=c'-t'\in A'(M')\subseteq A'(L')
$$ 
is as required by (\ref{eqn:goal}).
First of all note that indeed
$$
 \alpha(x')=\alpha(\alpha'(c)-\alpha'(t))=m\cdot(c-t)\in U(M)\subseteq U(L).
$$
Then choose $w_i\in W_i(\bar{K})$ with $\delta_i(w_i)=c'$
and let $z_i\in Z_i(\bar{K})$ be the image of $w_i$ in $Z_i$.
Then $\rho_i(z_i)=c'$,
so $z_i$ is the unique point in the fiber $\rho_i^{-1}(c')$.
Thus, if $D_i:=D(w_i/c')\leq H_i$ is the decomposition group of $w_i$ over $c'\in A'(K')$,
then the image of $D_i$ under ${\rm pr}\colon\Gamma_i\wr G\rightarrow G$ is all of $G$
(Lemma \ref{lem:decomposition_group}(d,b)), 
and $D_i\cap\Gamma_i^G$ surjects onto $\Gamma_i^{\{1,\sigma\}}$
(Lemma \ref{lem:decomposition_group}(c,b)).
Let $F_i=M'(w_i)$ be the residue field of $w_i$,
and $F_{i,1}$ the residue field of the image of $z_i$ in $Y_{i,1}$,
so that ${\rm Spec}(F_{i,1})=\pi_{i,1}^{-1}(c')$.
By 
Lemma \ref{lem:decomposition_group}(a,d)
we can identify
${\rm Gal}(F_i/K')= D_i\leq\Gamma_i\wr G$
and
${\rm Gal}(F_{i,1}/M')=\Gamma_i^{\{1\}}$
such that
the restrictions of ${\rm pr}$ and $e_1$ to $D_i$ respectively $D_i\cap\Gamma_i^G$
correspond to the restriction maps
${\rm Gal}(F_i/K')\rightarrow{\rm Gal}(M'/K')$
respectively ${\rm Gal}(F_i/M')\rightarrow{\rm Gal}(F_{i,1}/M')$.
As $L'/K'$ is abelian, 
$N_i:={\rm Gal}(F_i/F_i\cap L')$ contains
the commutator subgroup ${\rm Gal}(F_i/K')'=D_i'$.
Thus by Lemma~\ref{lem:gt_abelian}, $N_i\cap\Gamma_i^G$ surjects onto $\Gamma_i^{\{1\}}$,
which means that $F_{i,1}$
is linearly disjoint from $L'$ over $M'$,
hence
$$
 (\pi_i)_{L'}^{-1}(x') =
 (\pi_{i,1})_{L'}^{-1}(x'+t')=
 (\pi_{i,1})_{L'}^{-1}(c')=
 \pi_{i,1}^{-1}(c')\times_{{\rm Spec}(M')}{\rm Spec}(L')=
 {\rm Spec}(F_{i,1}\otimes_{M'}L')
$$
is irreducible.
\end{proof}

\begin{proof}[Proof of Corollary \ref{cor:intro_Frey_Jarden}]
This follows from Theorem \ref{thm:intro_abelian}
and Lemma \ref{lem:rank_res}.
\end{proof}

\begin{proof}[Proof of Corollary \ref{cor:intro_E}]
This follows from Theorem \ref{thm:intro_abelian},
Proposition \ref{prop:FreyJarden} and
Remark \ref{lem:rank}.
\end{proof}

In fact, we obtain the following generalization of Corollary \ref{cor:intro_E}:

\begin{corollary}\label{cor:hyper}
Let $A$ be a geometrically simple abelian variety over a finitely generated field $K$ of characteristic zero. 
Assume there exists a geometrically integral Galois cover $C\to \mathbb{P}^1_K$ with abelian Galois group $\mathrm{Gal}(C/\mathbb{P}^1)$ and an epimorphism $J_C\to A$, where $J_C$ denotes the Jacobian variety of $C$. 
Then $A_{K^{\rm ab}}$ has WHP.
\end{corollary}

\begin{proof} By \cite[Theorem 1.1]{Petersen},
$A$ has infinite rank over $K^{\mathrm{ab}}$,
hence since $A_{K^{\mathrm{ab}}}$ is simple,
so does every nonzero homomorphic image of $A_{K^{\mathrm{ab}}}$. 
Thus Theorem \ref{thm:intro_abelian} implies the assertion.
\end{proof}

\begin{example} 
Let $K$ be a finitely generated field of characteristic zero and
let $f\in K[X]$ be an irreducible polynomial of degree $n\ge 5$ 
with Galois group isomorphic to $S_n$ or $A_n$. 
Let $C$ be the smooth projective curve with affine equation $Y^2=f(X)$. 
Then $\mathrm{End}_{\bar{K}}((J_C)_{\bar{K}})=\mathbb{Z}$ by \cite[Theorem 2.1]{zarhin}. 
In particular, $J_C$ is geometrically simple. 
Moreover, $C$ is a geometrically integral Galois cover of $\mathbb{P}^1_K$ with Galois group isomorphic to $\mathbb{Z}/2\mathbb{Z}$. 
Corollary \ref{cor:hyper} implies that $(J_C)_{K^{\mathrm{ab}}}$ has WHP. 
\end{example}

\begin{remark}
Let $X$ and $Y$ be smooth proper $K$-varieties.
\cite[Corollary 3.4]{BFP} proves that if $X$ and $Y$ have HP,
then so does the product $X\times Y$.
If one would have a similar result for WHP, it would suffice
to prove Theorem \ref{thm:intro_abelian} in the special case that $A$ is simple, 
which would simplify the construction of the point $t$
(Proposition \ref{prop:t}) considerably.
In fact, \cite[Theorem 1.8]{CDJLZ} proves such a product theorem for WHP,
however only for $K$ finitely generated of characteristic zero,
which seems to be not sufficient here.
\end{remark}

\begin{remark}
By \cite[Proposition 3.5]{CDJLZ}, WHP is preserved under isogenies of abelian varieties, hence in the proof of Theorem \ref{thm:intro_abelian}
(and therefore in Proposition \ref{prop:t})
one could assume that
$\prod_{i=1}^kA_i\rightarrow A_L$
is the identity.
This however would not allow a significant simplification of the proof.
\end{remark}

\begin{remark}\label{rem:counterexample}
Let
$A$ be an abelian variety 
over a finitely generated field $K$ of characteristic zero,
and let
$(\pi_i \colon Y_i \to A)_{i=1}^n$ be
as in Definition \ref{def:WHP}.
The more precise form of Theorem \ref{thm:CZ}
proven in \cite[Theorem 1.3]{CDJLZ}
is that for every Zariski-dense subgroup $\Omega\leq A(K)$
there exists a finite index coset $C\subseteq\Omega$ 
disjoint from $\bigcup_{i=1}^n\pi_i(Y_i(K))$.
Such a statement however does in general not hold over $L=K^{\rm ab}$ instead of $K$:
Assume that $A(K)$ is Zariski-dense
and let $\pi\colon Y\rightarrow A$ be any ramified abelian cover of $A$.
(To see that such a cover exists, let for example $Y$ be the normalization of $A$
in the quadratic extension of
the function field $K(A)$
obtained by adjoining a square root of a rational function on $A$ whose principal divisor is not divisible by 2.)
Then $\Omega=A(K)$ is a Zariski-dense (finitely generated) subgroup of $A(L)$,
but $\pi^{-1}(\Omega)\subseteq Y(K^{\rm ab})$, so in fact
$\Omega\subseteq\pi(Y(L))$.
This also shows that if $X$ is any proper smooth  $K$-variety with $X(K)=X(K^{\rm ab})$, then $X_{K^{\rm ab}}$ does not have WHP:
If $\pi\colon Y\rightarrow X$ is any ramified abelian cover,
$\pi(Y(K^{\rm ab}))\supseteq X(K)=X(K^{\rm ab})$.
\end{remark}

\section{Torsion fields of abelian varieties}
\label{sec:tor}

\noindent
In this section we prove
Theorem \ref{thm:intro_E_tor}.
We denote by 
$\mu_n\subseteq\mathbb{C}$ the group of $n$-th roots of unity,
and by $\mathbb{P}$ the set of prime numbers.
For $p\in\mathbb{P}$ we 
denote by $v_p\colon\mathbb{Q}^\times\rightarrow\mathbb{Z}$
the $p$-adic valuation,
and we abbreviate
$\mu_{p^\infty}:=\bigcup_{k=1}^\infty\mu_{p^k}$
and
$A[p^\infty]:=\bigcup_{k=1}^\infty A[p^k]$.

\begin{lemma}\label{lem:tor_cycl}
Let
$p\in\mathbb{P}$ and
$A$ an abelian variety over a field $K$ of characteristic zero. Then
$K(\mu_{p^\infty})\subseteq K(A[p^\infty])$.
\end{lemma}

\begin{proof}
Let $L=K(A[p^\infty])$
and choose an isogeny $\lambda\in \mathrm{Hom}_K(A, A^\vee)$, 
where $A^\vee$ is the dual of $A$ (see \cite[\S10]{Mil86} or \cite[Theorem 6.18]{vdGM}). 
Then $\lambda(A[p^\infty])=A^\vee[p^\infty]$,
hence $K(A^\vee[p^\infty])\subseteq L$. 
Since for every $n\in\mathbb{N}$ there is a
non-degenerate pairing of $\Gal(K)$-modules
$$ 
 e_n\colon A[p^n]\times A^\vee[p^n]\to \mu_{p^n},
$$
see \cite[\S16]{Mil86} or \cite[Def.~11.11]{vdGM}, 
and ${\rm Gal}(L)$ acts trivially on $A[p^n]\times A^\vee[p^n]$,
we conclude that $\mu_{p^n}\subseteq L$.
\end{proof}

\begin{lemma}\label{lem:values_of_pol}
For every $f\in\mathbb{Z}[X]$
which is not a square in $\mathbb{C}[X]$,
there exists a finite set
$P_f\subseteq\mathbb{P}$
such that for every 
finite set $P\subseteq\mathbb{P}$
there are infinitely many $x\in\mathbb{Z}$
with
\begin{enumerate}[{\rm(a)}]
    \item $p\nmid f(x)$ for every $p\in P\setminus  P_f$, and
   \item $v_p(f(x))\equiv 1\mod 2$
for some  $p\in\mathbb{P}\setminus P$.
\end{enumerate}
\end{lemma}

\begin{proof}
Since $f\neq 0$,
the set  $P_f$ of $p\in\mathbb{P}$
with $f(x)\equiv 0\mbox{ mod }p$ for all $x\in\mathbb{Z}$
is finite.
Let $P\subseteq\mathbb{P}$  finite.
By the Chinese remainder theorem there exists $a\in\mathbb{Z}$
with $f(a)\not\equiv0\mbox{ mod }p$ for every $p\in P\setminus  P_f$.
Let %$e=\max_{p\in P_f}v_p(f(a))$,
%$b=\prod_{p\in P}p\cdot\prod_{p\in P_f}p^{e+1}$,
$b=\prod_{p\in P\setminus  P_f}p$
and $g(Y)=f(a+bY)$.
Then for every $y\in\mathbb{Z}$,
$g(y)\equiv f(a)\not\equiv 0\mbox{ mod }p$ for every $p\in P\setminus  P_f$, hence
$x=a+by$ satisfies (a).
The assumption on $f$ implies that $g$ is not a square in $\mathbb{C}[Y]$,
hence Hilbert's irreducibility theorem
in the form \cite[Theorem~3.4.4]{Serre} or 
\cite[Theorem~13.3.5]{FJ} applied to 
the $2^{|P|+1}$ polynomials
$$
 Z^2 \pm \left(\prod\nolimits_{p\in P_0}p\right)\cdot g(Y)\in\mathbb{Z}[Z,Y],\quad P_0\subseteq P
$$
gives
infinitely many $y\in\mathbb{Z}$
for which $c\cdot|g(y)|$ is not a perfect square for any $c=\prod_{p\in P_0}p$, $P_0\subseteq P$,
and then $x=a+by$ satisfies (b).
\end{proof}

\begin{lemma}\label{lem:gt_tor}
Let $\Gamma$ and $G$ be groups
and suppose 
${\rm pr}\colon\Gamma\wr G\rightarrow G$ factors as 
$$
 \Gamma\wr G\stackrel p\twoheadrightarrow H\stackrel{r}\twoheadrightarrow G.
$$ 
Let $H_1,H_2\unlhd H$
and write $G_i=r(H_i)$.
Assume that $G_1\neq 1$, $G_2\neq 1$, 
$G=G_1\times G_2$,
and $[H_1,H_2]= 1$.
Then ${\rm Ker}(p)\subseteq\Gamma^G$ surjects
onto $\Gamma$ under the
evaluation map $e_1$.
\end{lemma}

\begin{proof}
Let $N={\rm Ker}(p)$, $\Gamma_0=e_1(N)$, and
$\bar{\Gamma}:=\Gamma/\Gamma_0$.
The map $\Gamma\wr G\rightarrow\bar{\Gamma}\wr G$, $(f,g)\mapsto(\bar{f},g)$ where
$\bar{f}(h)=f(h)\Gamma_0$ is an epimorphism whose kernel contains $N$,
so without loss of generality we can assume that $N$ equals this kernel,
and thus $H=\bar{\Gamma}\wr G$.
Applying
\cite[Lemma 13.7.4(b)]{FJ}
with $G_0=1$, $A=\bar{\Gamma}$ and $1\neq h_2\in H_2$
shows that $\bar{\Gamma}=1$.
\end{proof}

\begin{proof}[Proof of Theorem \ref{thm:intro_E_tor}]
For each $p\in\mathbb{P}$
let $L_p=\mathbb{Q}(A[p^\infty])$.
Then 
$\mathbb{Q}(\mu_{p^\infty})\subseteq L_p$ by Lemma~\ref{lem:tor_cycl},
and $L:=\prod_{p\in\mathbb{P}}L_p=\mathbb{Q}(A_{\rm tor})$.
In particular, $\mathbb{Q}^{\rm ab}\subseteq L$ by the Kronecker--Weber theorem,
hence ${\rm rk}(E(L))=\infty$
by Proposition \ref{prop:FreyJarden}.
Like in that proposition,
choose an affine model 
$$ 
 E_0 : Y^2=f(X), \quad f\in\mathbb{Z}[X]
$$
of $E$,
and for $a\in\mathbb{Q}$ let $x_a\in E_0(\sqrt{f(a)})$
be the corresponding point.

Let $K$ be a number field contained in $L$,
$U\subseteq E_K$ nonempty and open, 
$K'/K$ a finite Galois extension
linearly disjoint from $L/K$,
$\alpha\colon E'\rightarrow E_{K'}$ an isogeny,
and $(\pi_i\colon Y_i\rightarrow E')_{i=1}^n$
a finite collection of fully ramified geometrically integral Galois covers
of $K'$-varieties.
By Lemma \ref{lem:red_to_irred} it suffices to find
a finite Galois extension $K''/K$ containing $K'$ and
\begin{equation}\label{eqn:goal2}
x'\in E'(K''L)\cap\alpha^{-1}(U(L))\;
\mbox{ with }\;(\pi_i)_{K''L}^{-1}(x')\mbox{ irreducible for every }i.
\end{equation}
The freedom to choose $K''$
allows us, using Lemma \ref{lem:fiber_products}(e), to freely replace
$K$ by a finite extension $K^*$ of $K$ inside $L$ (and accordingly $K'$ by $K^*K'$),
as well as to replace $K'$ by
a bigger finite Galois extension $K^\dagger$ of $K$
(and accordingly $K$ by $L\cap K^\dagger$).
Let $L'=K'L$ and $L'_p=K'L_p$.
Since ${\rm rk}(E(L))=\infty$,
by enlarging $K$ 
%(and accordingly $K'$, $L'$, $L'_p$)
we can assume without loss of generality
that ${\rm rk}(E(K))>0$,
and then $E(K)$ is Zariski-dense in $E$.
By Serre's independence theorem (see \cite[Thm.~1]{Serreindependence} 
and \cite{serre-indep2})
we may
assume, after possibly enlarging $K'$
further,
that the family $(L'_p)_{p\in\mathbb{P}}$
is linearly disjoint over $K'$.
Enlarging $K'$ 
once again if necessary,
we may assume without loss of generality
that $\sqrt{-1}\in K'$
and that 
$P':=\{p\in\mathbb{P}:\sqrt{p}\in K'\}$
contains $P_f\cup\{2\}$,
where $P_f$ is the set of primes from Lemma~\ref{lem:values_of_pol}.

Let $m={\rm deg}(\alpha)$ and let $\alpha'\colon E_{K'}\rightarrow E'$ be the isogeny
with $\alpha\circ\alpha'=[m]_{E_{K'}}$ (Lemma \ref{lem:av_dual_isog}).
The set 
$$
 S:=\bigcup_{i=1}^n\left\{x'-x'':x',x''\in{\rm Branch}(\pi_i)(\bar{K})\right\}\subseteq E'(\bar{K})
$$ 
is finite
and invariant under ${\rm Gal}({K'})$.
Applying Lemma \ref{lem:values_of_pol}
with $P=P'$ gives infinitely many $k_1\in\mathbb{Z}$ such that $\sqrt{f(k_1)}\notin K'$.
Since $[2]^{-1}{\alpha}'^{-1}(S)$
is finite,
$x_{k_1}\notin[2]^{-1}{\alpha}'^{-1}(S)$ for cofinitely many of these $k_1$,
and we fix such $k_1$
and let 
$M_1:=K(\sqrt{f(k_1)})$
and $t_1:=x_{k_1}\in E(M_1)$.
Applying Lemma \ref{lem:values_of_pol} again, now with 
$P=P'\cup\mathbb{P}(k_1)$, where
for $k\in\mathbb{Z}$ we write
$$
 \mathbb{P}(k) := \{p\in\mathbb{P}:v_p(f(k))\equiv 1\mod 2\},
$$ 
gives infinitely many
$k_2\in\mathbb{Z}$
for which 
$\sqrt{f(k_2)}\notin K'$ and
$\mathbb{P}(k_1)\cap\mathbb{P}(k_2)\subseteq P_f\subseteq P'$.
We fix such $k_2$ for which
$$
 x_{k_2}\notin [2]^{-1}{\alpha}'^{-1}(S)\cup\tau_{-t_1}([2]^{-1}{\alpha}'^{-1}(S))
$$ 
and let $M_2:=K(\sqrt{f(k_2)})$ and $t_2:=x_{k_2}\in E(M_2)$.
Let $P_1=\mathbb{P}(k_1)$
and $P_2=\mathbb{P}\setminus P_1$.
Then,
since $\mathbb{P}(k_2)\subseteq P_2\cup P'$
and $\sqrt{-1},\sqrt{2}\in K'$,
we have for $i=1,2$ that
$$
 M_i':=K'M_i=K'(\sqrt{|f(k_i)|})
 \subseteq K'(\sqrt{ p}:p\in\mathbb{P}(k_i)\setminus P')
 \subseteq K'(\mu_{p}:p\in P_i)\subseteq N_i := \prod_{p\in P_i}L'_p.
$$ 
Note that $N_1$ and $N_2$ are linearly disjoint Galois extensions of $K'$ with $N_1N_2=L'$,
hence 
\begin{equation}\label{eqn:Gal_N}
 {\rm Gal}(L'/K')={\rm Gal}(L'/N_1)\times{\rm Gal}(L'/N_2).
\end{equation}
In particular, also $M_1'$ and $M_2'$ are linearly disjoint over $K'$.
Let $M:=M_1M_2$ and $M':=K'M=M_1'M_2'$,
observe that
\begin{equation}\label{eqn:M}
 G := {\rm Gal}(M'/K')={\rm Gal}(M'/M_1')\times{\rm Gal}(M'/M_2')\cong C_2\times C_2
\end{equation}
is a Klein four-group,
and let 
$1\neq \sigma_i\in{\rm Gal}(M'/M_i')$
for $i=1,2$.
$$
 
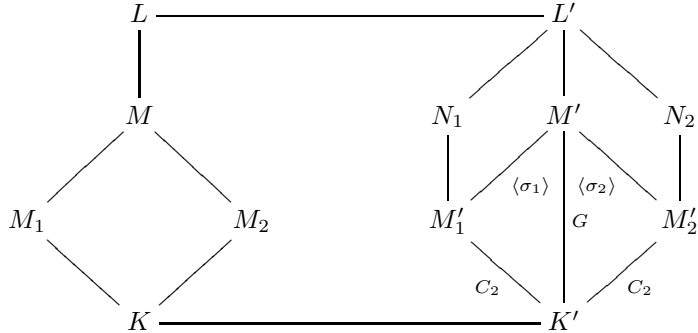
\xymatrix{
 & L\ar@{-}[rrrr]\ar@{-}[d] &&&& L'\ar@{-}[d]\ar@{-}[dl]\ar@{-}[dr] && \\
 & M\ar@{-}[dl]\ar@{-}[dr]  &&&N_1\ar@{-}[d]& M'\ar@{-}[dl]^{\left<\sigma_1\right>}\ar@{-}[dr]_{\left<\sigma_2\right>}\ar@{-}[dd]^G & N_2\ar@{-}[d]& \\
 M_1\ar@{-}[dr]&   &M_2\ar@{-}[dl]&& M_1'\ar@{-}[dr]_{C_2} && M_2'\ar@{-}[dl]^{C_2} \\
 & K\ar@{-}[rrrr] &&&& K' \\
 }
$$
Note that $\sigma_i(t_j)=(-1)^{i-j}t_j$
for $i,j\in\{1,2\}$.
Let $t:=t_1+t_2\in E(M)$
and $t':=\alpha'(t)\in E'(M')$ 
and observe that
$$
 t'-\sigma_1(t')=\alpha'(t_1+t_2)-\sigma_1(\alpha'(t_1+t_2))=2\alpha'(t_2)\notin S,
$$ 
analogously $t'-\sigma_2(t')=2\alpha'(t_1)\notin S$.
Moreover,
$$
 t'-\sigma_1\sigma_2(t') = \alpha'(2t_1+2t_2)\notin S.
$$
This implies
that $g(t')-h(t')\notin S$ for every $g,h\in G=\{1,\sigma_1,\sigma_2,\sigma_1\sigma_2\}$ with $g\neq h$.

For $g\in G$,
$$
 \pi_{i,g}:=\tau_{g(t')}\circ(\pi_i)_{M'}\colon Y_{i,g}=(Y_i)_{M'}\rightarrow E'_{M'}
$$ 
is a geometrically integral fully ramified Galois cover of $M'$-varieties.
Note that $\pi_{i,g}=(\pi_{i,1})^g$,
which induces a natural isomorphism
$\Gamma_i:={\rm Gal}(Y_{i,1}/E'_{M'})\rightarrow {\rm Gal}(Y_{i,g}/E'_{M'})$
(Remark~\ref{rem:gX}).
Moreover,
${\rm Branch}(\pi_{i,g})\cap{\rm Branch}(\pi_{i,h})=\emptyset$
for every $g,h\in G$ with $g\neq h$,
since otherwise there exist $x',x''\in{\rm Branch}(\pi_i)(\bar{K})$
with $x'+g(t')=x''+h(t')$,
leading to the contradiction $g(t')-h(t')=x''-x'\in S$.
Thus by Lemma \ref{lem:lin_disj_fully_ram} (with $n=4$)
the normalization of $E'_{M'}$ in the compositum $\prod_{g\in G}M'(Y_{i,g})$
is a fully ramified geometrically integral Galois cover
$\rho_i\colon Z_i\rightarrow E'_{M'}$ with Galois group
$$
 {\rm Gal}(Z_i/E'_{M'})=\prod_{g\in G}{\rm Gal}(Y_{i,g}/E'_{M'})=\Gamma_i^G,
$$
so
by Lemma \ref{lem:fiber_product_galois}
the composition $\delta_i\colon Z_i\rightarrow E'_{M'}\rightarrow E'$
is a Galois cover whose Galois group
we can identify with $\Gamma_i\wr G$
such that
the maps $\mathrm{pr}$ and $e_g$, for $g\in G$, 
coincide with the restriction maps $\Gal(Z_i/E')\to G$ and $\Gal(Z_i/E'_{M'})\to \Gal(Y_{i,g}/E'_{M'})=\Gamma_i$. 
$$
 \xymatrix{
    &&& & Z_i\ar[ddll]^{\Gamma_i^G\rtimes G}_{\delta_i}\ar[dd]^{\Gamma_i^G}_{\rho_i}\ar[dr] \\
    E_{K'}\ar[d]_{[m]}\ar[drr]^{\alpha'} &&Y_i\ar[d]_{\pi_i}^{\Gamma_i}&  && Y_{i,1}\ar[dl]^{\pi_{i,1}} \\ 
   E_{K'}&&E'\ar[ll]^\alpha && E'_{M'}\ar[ll]_G
 }
$$
Let $\Omega:=E(K)$ and
$$
 \Omega':=\alpha'( \Omega )\subseteq E'(K')\subseteq E'({M'}).
$$ 
Since $\Omega$ is a Zariski-dense subgroup of $E$
and $\alpha'$ is an isogeny, $\Omega'$ is a Zariski-dense 
subgroup of $E'$.
Thus by 
\cite[Theorem 1.4]{CDJLZ} applied to $E'_{M'}$ and $\Omega'$,
there exists 
a finite index coset $C'\subseteq\Omega'$ such that
for every $c'\in C'$ each of the fibers $\rho_i^{-1}(c')$ is integral.
Then $C:=\Omega\cap\alpha'^{-1}(C')$ is a finite index coset of $\Omega$,
hence Zariski-dense in $E$,
and so we can pick
$c\in C\cap\tau_{t}([m]^{-1}(U))$
and let $c':=\alpha'(c)$.
Without loss of generality, each $\delta_i$ is unramified over $c'$.

We claim that 
$$
 x':=c'-t'\in E'(M')\subseteq E'(L')
$$ 
is as required by (\ref{eqn:goal2}).
First of all note that indeed
$$
 \alpha(x')=\alpha(\alpha'(c)-\alpha'(t))=m\cdot(c-t)\in U(M)\subseteq U(L).
$$
Then let $z_i\in Z_i(\bar{K})$ with $\rho_i(z_i)=c'$,
let $y_i$ be the image of $z_i$ in $Y_{i,1}$,
and let $F_i=M'(z_i)$
and $F_{i,1}=M'(y_i)$ be the residue fields.
Since $\rho_i^{-1}(c')$ is integral,
we have $\rho_i^{-1}(c')={\rm Spec}(F_i)$ and $\pi_{i,1}^{-1}(c')={\rm Spec}(F_{i,1})$,
and by Lemma \ref{lem:decomposition_group}
we can identify
${\rm Gal}(F_i/K')=D(z_i/c')=\Gamma_i\wr G$
so that
${\rm pr}$ and $e_1$ correspond to the restriction maps
${\rm Gal}(F_i/K')\rightarrow{\rm Gal}(M'/K')$ respectively
${\rm Gal}(F_i/M')\rightarrow{\rm Gal}(F_{i,1}/M')$.
For $j=1,2$ let $H_j:={\rm Gal}(F_i\cap L'/F_i\cap N_j)$
and let $r\colon{\rm Gal}(F_i\cap L'/K')\rightarrow{\rm Gal}(M'/K')$ denote
the restriction map.
Then (\ref{eqn:Gal_N}) implies that $H_1H_2={\rm Gal}(F_i\cap L'/K')$ and $[H_1,H_2]=1$.
Since $N_j\cap M'=M_j'$, we get from (\ref{eqn:M})
that $r(H_j)={\rm Gal}(M'/M_j')\cong C_2$ for each $j$ and 
$G=r(H_1)\times r(H_2)$.
Therefore, Lemma \ref{lem:gt_tor} implies that
${\rm Gal}(F_i/F_i\cap L')\leq\Gamma_i^G$
surjects onto $\Gamma_i^{\{1\}}={\rm Gal}(F_{i,1}/M')$,
hence
$F_{i,1}$ and $L'$ are linearly disjoint over $M'$.
Thus
$$
 (\pi_i)_{L'}^{-1}(x') = %(\tau_{-t'}\circ\pi_{i,1})_{L'}^{-1}(x')=
 (\pi_{i,1})_{L'}^{-1}(x'+t')=(\pi_{i,1})_{L'}^{-1}(c')=
 \pi_{i,1}^{-1}(c')\times_{{\rm Spec}(M')}{\rm Spec}(L')=
 {\rm Spec}(F_{i,1}\otimes_{M'}L')
$$ 
is irreducible, concluding the proof.
\end{proof}

\begin{remark}
While \cite{Jarden} proved that
$\mathbb{Q}(A_{\rm tor})$ is Hilbertian,
the series of works
\cite{FJP,FP_div,Thornhill,BFW}
established Jarden's conjecture \cite[Conjecture 1]{Jarden}
that 
for every Hilbertian field $K$
and every abelian variety $A/K$
(and then in fact every commutative algebraic group, with some exceptions in positive characteristic), 
every intermediate field
$K\subseteq L\subseteq K(A_{\rm tor})$ is Hilbertian.
We note that our proof 
that $E_{\mathbb{Q}(A_{\rm tor})}$ has WHP does not seem to generalize to $E_L$ for arbitrary intermediate fields
$\mathbb{Q}\subseteq L\subseteq \mathbb{Q}(A_{\rm tor})$,
which however would have been surprising in light of Remark \ref{rem:counterexample}.
\end{remark}

\begin{remark}
It is known that over a Hilbertian PAC field,
every geometrically integral variety has HP \cite[Prop.\ 27.3.4, Example 24.8.5(b)]{FJ}.
We note that although $\mathbb{Q}^{\rm ab}$ and $\mathbb{Q}(A_{\rm tor})$ are
Hilbertian, they are not PAC.
Indeed, for $\mathbb{Q}^{\rm ab}$ this is \cite[Cor.~11.5.7]{FJ},
and we sketch the proof for $\mathbb{Q}(A_{\rm tor})$:
If $p$ is a prime of good reduction of $A$,
then for every prime number $\ell\neq p$,
$\mathbb{Q}(A[\ell^\infty])$ is contained in 
the maximal unramified extension $L^{\rm ur}$ of the local field $L=\mathbb{Q}_p(A[p])$
by the N\'eron--Ogg--Shafarevich criterion \cite[Ch 7.4 Thm.~5]{BLR},
hence ${\rm Gal}(L^{\rm ur}(A_{\rm tor})/L^{\rm ur})$ 
is a subgroup of the pro-$p$ group
${\rm Ker}({\rm GL}_{A[p^\infty]}(\mathbb{Z}_p)\rightarrow{\rm GL}_{A[p]}(\mathbb{F}_p))$
coming from the action on the Tate module of $A$.
In particular, $\mathbb{Q}(A_{\rm tor})\mathbb{Q}_p=L(A_{\rm tor})$ is not algebraically closed,
which by a result of Frey--Prestel \cite[Cor.~11.5.5]{FJ}
implies that $\mathbb{Q}(A_{\rm tor})$ is not PAC.
\end{remark}

\section*{Acknowledgements}

\noindent
The authors would like to thank
Daniele Garzoni for helpful discussions around \cite{CDJLZ}, Cornelius Greither for interesting discussions around Lemma \ref{lem:fiber_products}, 
Remy van Dobben de Bruyn for the suggestion to use the Hilbert scheme in the proof of Lemma~\ref{lem_stab},
and the referee as well as Jakob Stix for helpful remarks on the submitted version.

Part of this work was done while A.F.\ was a guest of Tel Aviv University,
and he would like to thank the School of Mathematics for their hospitality.
L.B.-S. was supported by the Israel Science Foundation (grant no.~702/19).
S.P. was supported by a research grant UMO-2018/31/B/ST1/01474 of the National Centre of Sciences of Poland.

%\section*{Data availability statement and
%conflict of interest statement}
%
%On behalf of all authors, the corresponding author states that there is no conflict of interest.
%Data sharing is not applicable to this article as no datasets were generated or analysed during the current study.


\begin{thebibliography}{CDJLZ20}

\bibitem[BF13]{BFsurvey}
L.~Bary-Soroker and A.~Fehm.
\newblock Open problems in the theory of ample fields.
\newblock In {\em Geometric and differential {G}alois theories}, volume~27 of {\em S\'emin. Congr.}, pages 1--11. Soc. Math. France, Paris, 2013.

\bibitem[BFP14]{BFP}
L.~Bary-Soroker, A.~Fehm and S.~Petersen.
\newblock On varieties of Hilbert type.
\newblock {\em Ann.\ Inst.\ Fourier, Grenoble} 64(5):1893--1901, 2014.

\bibitem[BFW16]{BFW}
L.~Bary-Soroker, A.~Fehm and G.~Wiese.
\newblock Hilbertian fields and Galois representations.
\newblock {\em J.\ reine angew.\ Math.} 712:123--139, 2016.

\bibitem[BG22]{BG}
L.~Bary-Soroker and D.~Garzoni.
\newblock Hilbert's irreducibility theorem via random walks.
\newblock To appear in {\em Int. Math. Res. Not.} 

\bibitem[BHP20]{BHP}
M.~Bays, B.~Hart and A.~Pillay.
\newblock Universal covers of commutative finite Morley rank groups.
\newblock {\em J.\ Inst.\ Math.\ Jussieu} 19(3):767--799, 2020.

\bibitem[BMMN98]{BMMN}    
M.~Bhattacharjee, D.~Macpherson, R.~G.~Möller and P.~M.~Neumann.
\newblock {\em  Notes on Infinite Permutation Groups.}
\newblock Springer, 1998.

\bibitem[Bor15]{Borovoi}
M.~Borovoi.
\newblock Homogeneous spaces of Hilbert type.
\newblock {\em Int.\ J.\ Number Theory} 11(2):397--405, 2015.

\bibitem[BLR90]{BLR}
S.\ Bosch, W.\ L\"utkebohmert and M.\ Raynaud.
\newblock {\em N\'eron Models}.
\newblock Springer, 1990.

\bibitem[Bou06]{Bou}
N. Bourbaki.
\newblock {\em Algebr\'e commutative}.
\newblock {Springer}, 2006.

\bibitem[Coc19]{Coccia}
S.~Coccia.
\newblock The Hilbert property for integral points of affine smooth cubic surfaces. 
\newblock {\em J.\ Number Theory} 200:353--379, 2019.

\bibitem[CTS87]{CTS}
J.-L.~Colliot-Th\'el\`ene and J.-J.~Sansuc.
\newblock Principal homogeneous spaces under flasque tori: applications.
\newblock {\em J.~Algebra}, 106:148--205, 1987.

\bibitem[CDJLZ22]{CDJLZ}
P.~Corvaja, J.~L.~Demeio, A.~Javanpeykar, D.~Lombardo and U.~Zannier.
\newblock On the distribution of rational points on ramified covers of
abelian varieties.
\newblock {\em Compositio Math.} 158(11):2109--2155, 2022.

\bibitem[CZ17]{CZ}
P.~Corvaja and U.~Zannier.
\newblock On the Hilbert property and the fundamental group of algebraic varieties.
\newblock {\em Math.\ Z.} 286:579--602, 2017.


\bibitem[Dem20]{Demeio}
J.~L.~Demeio.
\newblock Non-rational varieties with the Hilbert Property.
\newblock {\em Int.\ J.\ Number Theory} 16(4):803--822, 2020.

\bibitem[Dem21]{Demeio2}
J.~L.~Demeio.
\newblock Elliptic fibrations and the Hilbert property.
\newblock {\em Int.\ Math.\ Research Notices}, Volume 2021, Issue 13, July 2021, Pages 10260--10277.

\bibitem[DZ07]{DZ}
R.~Dvornicich and U.~Zannier.
\newblock Cyclotomic Diophantine problems (Hilbert irreducibility and invariant sets for polynomial maps).
\newblock {\em Duke Math.\ J.} 139(3):527--554, 2007.

\bibitem[EGM]{vdGM}
B.~Edixhoven, G.~van der Geer and B.~Moonen.
\newblock {\em Abelian varieties}.
\newblock Manuscript, available at
\href{http://van-der-geer.nl/~gerard/AV.pdf}
{http://van-der-geer.nl/$\sim$gerard/AV.pdf}

\bibitem[FGA2]{FGA2}
B. Fantechi, L. Göttsche, L. Illusie, S. Kleiman, N. Nisure, and A. Vistoli.
\newblock {\em Fundamental algebraic geometry: Grothendieck's FGA explained}.
\newblock Mathematical Surveys and Monographs, Volume 123, AMS, 2005.


\bibitem[FJP12]{FJP}
A.~Fehm, M.~Jarden and S.~Petersen.
\newblock Kuykian fields.
\newblock {\em Forum Math.} 24:1013--1022, 2012.

\bibitem[FP10]{FehmPetersen}
A.~Fehm and S.~Petersen.
\newblock On the rank of abelian varieties over ample fields.
\newblock {\em Int.\ J.\ Number Theory} 6(3):579--586, 2010.

\bibitem[FP13]{FP_div}
A.~Fehm and S.~Petersen.
\newblock Hilbertianity of division fields of commutative algebraic groups.
\newblock {\em Israel J.\ Math.} 195(1):123--134, 2013

\bibitem[FJ74]{FreyJarden}
G.~Frey and M.~Jarden.
\newblock  Approximation theory and the rank of abelian varieties over large algebraic fields.
\newblock {\em Proc.~London Math.~Soc.} s3-28:112--128, 1974.

\bibitem[FJ08]{FJ}
M.~D.~Fried and M.~Jarden.
\newblock {\em Field Arithmetic}.
\newblock Third Edition. Springer, 2008.


%\bibitem[EGA1]{EGAI}
%A.~Grothendieck.
%\newblock {\'El\'ements de g\'eom\'etrie alg\'ebrique (r\'edig\'e avec la
%  cooperation de Jean Dieudonn\'e): I. Le langage des sch\'emas}.
%\newblock {\em {Publ. Math. IHES}}, ({4}):5--228, 1960.

\bibitem[EGA2]{EGAII}
A.~Grothendieck.
\newblock {\'El\'ements de g\'eom\'etrie alg\'ebrique (r\'edig\'e avec la
  cooperation de Jean Dieudonn\'e): II. \'Etude globale \'el\'ementaire de
  quelques classes de morphismes}.
\newblock {\em {Publ. Math. IHES}}, ({8}):5--222, 1961.

\bibitem[EGA4.1]{EGAIV1}
A.~Grothendieck.
\newblock {\'El\'ements de g\'eom\'etrie alg\'ebrique (r\'edig\'e avec la
  cooperation de Jean Dieudonn\'e): IV. \'Etude locale des sch\'emas et des
  morphismes des sch\'emas, Premiere partie}.
\newblock {\em {Publ. Math. IHES}} ({20}):{5--259}, 1964.


\bibitem[EGA4.2]{EGAIV2}
A.~Grothendieck.
\newblock {\'El\'ements de g\'eom\'etrie alg\'ebrique (r\'edig\'e avec la
  cooperation de Jean Dieudonn\'e): IV. \'Etude locale des sch\'emas et des
  morphismes des sch\'emas, Seconde partie}.
\newblock {\em {Publ. Math. IHES}} ({24}):{5--231}, 1965.


\bibitem[EGA4.4]{EGAIV4}
A.~Grothendieck.
\newblock {\'El\'ements de g\'eom\'etrie alg\'ebrique (r\'edig\'e avec la
  cooperation de Jean Dieudonn\'e): IV. \'Etude locale des sch\'emas et des
  morphismes des sch\'emas, Quatri\`eme partie}.
\newblock {\em {Publ. Math. IHES}} ({32}):{5--361}, 1967.

\bibitem[SGA1]{SGA1}
A.~Grothendieck.
\newblock {\em {S\'eminaire de G\'eom\'etrie Alg\'ebrique du Bois-Marie. SGA1 - Rev\^etements
  \'etales et groupe fondamental.}}
\newblock Springer, Lecture Notes in Mathematics 224, 1971.

\bibitem[FGA]{FGA}
A.~Grothendieck.
\newblock {\em {Techniques de construction et théorèmes  d’existence en
géométrie algébrique IV : les schémas de Hilbert.}}
\newblock Séminaire N. Bourbaki, 1961, exp. no 221, p. 249--276.

\bibitem[Hab13]{Habegger}
P.~Habegger.
\newblock Small height and infinite nonabelian extensions.
\newblock {\em Duke Math.\ J.} 162(11):2027--2076, 2013.

\bibitem[Har99]{Haran}
D.~Haran.
\newblock Hilbertian fields under separable algebraic extensions.
\newblock {\em Invent. Math.}, 137(1):113--126, 1999.

\bibitem[IL13]{ImLarsen}
B.-H.~Im and M.~Larsen.
\newblock Infinite rank of elliptic curves over $\mathbb{Q}^{\rm ab}$.
\newblock {\em Acta Arithmetica} 158(1):49--59, 2013.

\bibitem[Jar10]{Jarden}
M.~Jarden.
\newblock Diamonds in torsion of abelian varieties.
\newblock {\em J.\ Inst.\ Math.\ Jussieu} 9(3):477--480, 2010.

\bibitem[Jav21]{Javanpeykar}
A.~Javanpeykar.
\newblock Rational points and ramified covers of products of two elliptic curves.
\newblock {\em Acta Arith.} 198(3):275--287, 2021.

\bibitem[Jav22]{Javanpeykar22}
A.~Javanpeykar.
\newblock Hilbert irreducibility for varieties with a nef tangent bundle.
\newblock arXiv:2204.12828 [math.AG]

\bibitem[Kuy70]{Kuyk}
W.~Kuyk.
\newblock Extensions de corps hilbertiens.
\newblock {\em J.~Algebra} 14:112--124, 1970.

\bibitem[Lan62]{Lang_diophantine_geometry}
S.~Lang.
\newblock {\em Diophantine Geometry}.
\newblock John Wiley and Sons, 1962.

\bibitem[Lan83]{Lang_fundamentals}
S.~Lang.
\newblock {\em Fundamentals of diophantine geometry}.
\newblock Springer, 1983.

\bibitem[Lar05]{Larsen}
M.~Larsen.
\newblock A Mordell-Weil theorem for abelian varieties over fields generated by torsion points.
\newblock arXiv:math/0503378 [math.NT]

\bibitem[Mel95]{Meldrum}
J.~D.~P.~Meldrum.
\newblock {\em Wreath products of groups and semigroups}.
\newblock Longman, 1995.

\bibitem[Mil86]{Mil86}
J. Milne. 
\newblock Abelian Varieties.
\newblock  In G. Cornell and J. H. Silverman, editors, {\em Arithmetic Geometry, Proceedings of Storrs Conference.} Springer, 1986. 


\bibitem[Mil17]{Mil17}
J. Milne. 
\newblock Algebraic Groups. The Theory of Group schemes of Finite
Type over a field.
\newblock Cambridge Studies in Advanced Mathematics 170, 2017. 

\bibitem[Mum70]{Mum70}
D. Mumford. 
\newblock {\em Abelian Varieties}.
\newblock  Oxford University Press, 1970.

\bibitem[NS20]{NakaharaStreeter}
M.~Nakahara and S.~Streeter.
\newblock Weak approximation and the Hilbert property for Campana points.
\newblock arXiv:2010.12555v1  [math.NT].

\bibitem[Pet06]{Petersen}
S.~Petersen.
\newblock  On a question of Frey and Jarden about the rank of abelian varieties. (With an appendix by M. Jarden),
\newblock {\em J.~Number Theory} 120:287--302, 2006.


\bibitem[SY12]{SY}
F.~Sairaiji and T.~Yamauchi.
\newblock The rank of Jacobian varieties over the maximal abelian extensions of number fields: towards the Frey–Jarden conjecture.
\newblock {\em Canadian Math.\ Bull.} 55(4):842--849, 2012.

\bibitem[Ser86]{Serreindependence}
J.-P.~Serre.
\newblock R\'esum\'e des cours de 1985-1986.
\newblock Annuaire du Coll\`ege de France, Paris, 1986.

\bibitem[Ser08]{Serre}
J.-P.~Serre.
\newblock {\em Topics in Galois theory}.
\newblock Second edition. Springer, 2008.

\bibitem[Ser13]{serre-indep2}
J.-P.~Serre.
\newblock {Un crit\`ere d'ind\'ependance pour une famille de repr\'esentations
  $\ell$-adiques.}
\newblock {\em Commentarii Mathematici Helvetici} 88:543--576, 2013.

\bibitem[Stacks]{Stacks}
The {Stacks Project Authors}.
\newblock {\textit{Stacks Project}}, \url{https://stacks.math.columbia.edu}, 2023.

\bibitem[Str21]{Streeter}
S.~Streeter.
\newblock Hilbert property for double conic bundles and del Pezzo varieties.
\newblock {\em Math.\ Research Letters} 28(1):271--283, 2021.


\bibitem[Tho13]{Thornhill}
C.~Thornhill.
\newblock Abelian varieties and Galois extensions of Hilbertian fields.
\newblock {\em J.\ Inst.\ Math.\ Jussieu} 12(2):237--247, 2013.

% \bibitem[vDob22]{vDob2022}
% R. van Dobben de Bruyn.
% \newblock Stabilizers in abelian varieties are also abelian? reference request. 
% \newblock Answer to MathOverFlow question, \url{https://mathoverflow.net/q/419537}.


\bibitem[Zan10]{Zannier}
U.~Zannier.
\newblock Hilbert irreducibility above algebraic groups.
\newblock {\em Duke Math.\ J.} 153(2):397--425, 2010.

\bibitem[Zar00]{zarhin}
Y.~Zarhin.
\newblock Hyperelliptic Jacobians without complex multiplication.
\newblock {\em Mathematical Research Letters} 7(1):123--132, 2000.
\end{thebibliography}
\end{document}